\documentclass{article}
\usepackage{amsmath,amsfonts,amssymb,dsfont,mathrsfs}
\usepackage{amsthm}
\usepackage[dvipsnames]{xcolor}
\usepackage{caption}
\usepackage{url}
\usepackage{adjustbox}
\usepackage{lscape}
\usepackage{rotating}
\usepackage{hyperref}
\usepackage[numbers,sort&compress]{natbib}

\usepackage{tikz}
\usepackage{tikz-3dplot}
\usetikzlibrary{calc,,angles,quotes}

\usepackage{subcaption}

\usepackage[linesnumbered,ruled]{algorithm2e}
\SetKwComment{Comment}{/* }{ */}
\SetKwInOut{Input}{Input}
\SetKwInOut{Output}{Output}
\providecommand{\U}[1]{\protect\rule{.1in}{.1in}}

\newtheorem{theorem}{Theorem}[section]
\newtheorem{corollary}[theorem]{Corollary}
\newtheorem{lemma}[theorem]{Lemma}
\newtheorem{proposition}[theorem]{Proposition}

\newtheorem{remark}[theorem]{Remark}
\theoremstyle{definition}
\newtheorem{example}[theorem]{Example}

\newcommand{\R}{\mathbb{R}}

\newcommand{\N}{\mathbb{N}}

\newcommand{\Ball}{\mathbb{B}}
\newcommand{\E}{\mathbb{E}}
\newcommand{\haus}{\lambda}
\newcommand{\diam}{\mathrm{diam}}

\newcommand{\spn}{\mathrm{span}}

\newcommand{\inte}{\mathrm{int}}
\newcommand{\Lip}{\mathrm{Lip}}
\renewcommand{\P}{\mathcal{P}}
\newcommand{\CC}{\mathcal{C}}
\newcommand{\sph}{\mathbb{S}}
\newcommand{\D}{\mathcal{D}}
\newcommand{\ri}{\mathrm{ri}}
\newcommand{\calm}{\mathrm{calm}}

\providecommand{\argmin}{\mathop{\rm argmin}}

\providecommand{\proj}{\mathop{\rm proj}}
\providecommand{\conv}{\mathop{\rm conv}}

\providecommand{\tto}{\mathop{\rightrightarrows}\nolimits}

\newcommand{\dTV}{d_\mathrm{TV}}
\newcommand{\dW}{d_{W_1}}

\title{Lipschitz continuity of expected value under decision-dependent uncertainty with moving support}

\author{
	John Cotrina \thanks{ Universidad del Pac\'ifico, Lima, Per\'u. {\tt cotrina\_je@up.edu.pe} }
    \and
    Gonzalo Flores \thanks{Unidad de Acompañamiento Estudiantil, Universidad de O'Higgins, Rancagua, Chile. {\tt gonzalo.flores@uoh.cl}}
	\and 
	David Salas \thanks{Instituto de Ciencias de la Ingenier\'ia, Universidad de O'Higgins, Rancagua, Chile. {\tt david.salas@uoh.cl}}
	\and
	Anton Svensson \thanks{Instituto de Ciencias de la Ingenier\'ia, Universidad de O'Higgins, Rancagua, Chile. {\tt anton.svensson@uoh.cl}}
}
\date{}
\begin{document}
	
	\maketitle
	
\begin{abstract}
This paper addresses the problem of stochastic optimization with decision-dependent uncertainty, a class of problems where the probability distribution of the uncertain parameters is influenced by the decision-maker's actions. While recent literature primarily focuses on solving or analyzing these problems by directly imposing hypotheses on the distribution mapping, we explore in this work some of these properties for a specific construction by means of the moving support and a density function. The construction is motivated by the Bayesian approach to bilevel programming, where the response of a follower is modeled as the uncertainty, drawn from the moving set of optimal responses, which depends on the leader's decision. 
Our main contribution is to establish sufficient conditions for the Lipschitz continuity of the expected value function. We show that Lipschitz continuity can be achieved when the moving support is a Lipschitz continuous set-valued map with full-dimensional, convex, compact values, or when it is the solution set of a fully linear parametric problem. We also provide an example showing that the sole Lipschitz assumption on the moving set itself is not sufficient and that additional conditions are necessary.
\end{abstract}
	
(Set-valued and variational analysis).

\section{Introduction}\label{sec:intro}

Stochastic optimization with decision-dependent uncertainty is a variant of the usual stochastic programming problem, in which the decision-maker's actions influence the behavior of the uncertainty. Formally, for a feasible set $X$ one considers a mapping $\beta:x\in X\mapsto \beta_x\in\mathcal{P}(Y)$, where $\mathcal{P}(Y)$ denotes  the space of (Borelian) probability measures over some space $Y$. Then, for a cost function $\theta:X\times Y\to \R$, the archetypic problem is given by minimizing the expected value of $\theta(x,\xi)$, considering that $\xi$ is a random variable distributing with law $\beta_x$. That is, the standard problem is given by
\begin{equation}\label{eq:Decision-Dependent-Uncertainty}
    \min_{x\in X} \E_{\beta_x}[\theta(x,\cdot)]:=\int\theta(x,\xi) d\beta_x(\xi).
\end{equation}

The concept of stochastic optimization with decision-dependent distributions can be traced back to the  90's and the early 2000, with \cite{Jonsbraten1998:Class,ahmed2000strategic,GoelGrossmann2006:Class,Mallozzi1996,MallozziMorgan2005} (see also the references therein). The setting has resurfaced with its clear application to learning and data-driven optimization \cite{DrusvyatskiyXiao2023:Optimization,CutlerDrusvyatskiyHarchaoui2023:Drift,CutlerDiazDrusvyatskiy2024:Approximation,Nohadani2018:Optimization,Wang2025:Constrained, HeBolognaniDorflerMuehlebach2025Dynamics,WoodDallanese2023:StochasticSaddle}. In this literature, the main focus is on methodologies to solve Problem~\eqref{eq:Decision-Dependent-Uncertainty} (and its variants) based on direct hypotheses over the distribution mapping $\beta:x\mapsto \beta_x$.

Recently Problem~\eqref{eq:Decision-Dependent-Uncertainty} has been studied in \cite{salas2023existence}, in the context of stochastic bilevel programming \cite{BeckLjubicSchmidt2023Survey,BurtscheidtClaus2020BilevelLinear}. The idea is to consider a set-valued map $S:X\tto Y$ such that $S(x)$ models the uncertainty set over which the response of a second agent (the follower) is drawn. Then, by endowing the response $y\in S(x)$ with a decision-dependent probability distribution $y\sim \beta_x$, the problem of the decision-maker (the leader) is of the form of Problem~\eqref{eq:Decision-Dependent-Uncertainty}. The caveat here is that the probability measure $\beta_x$ must concentrate over the moving set $S(x)$, thus having moving support. Adopting the nomenclature of \cite{salas2023existence}, we will call the probability valued-map $\beta:x\mapsto \beta_x$ a \emph{belief} over the set-valued map $S$.

The model above was first introduced by Mallozzi and Morgan in 1996 \cite{Mallozzi1996} under the name of Intermediate Stackelberg games, since the main focus was to give an alternative between the optimistic and pessimistic approaches in bilevel optimization (see, e.g., \cite{dempe2002foundations,dempe2018pessimistic}).  The authors in \cite{salas2023existence} called it the Bayesian approach for bilevel games, since the main focus was to interpret the probability-valued map $\beta:x\mapsto \beta_x$, the belief, as a model by the leader of the uncertain behavior of the follower. There are some related works in the literature of stochastic bilevel programming but where the uncertainty is directly put on the data of the lower-level problem as a random variable $\xi(\omega)$, inducing a single-valued random response $y(x,\xi(\omega))$ (see, e.g., \cite{Claus2021Continuity,Claus2022Existence,BuchheimHenkeIrmai2022Knapsack}). The difficulty in such models is not the underlying distribution (which is constant) but rather the computation of the response map $y(x,\xi(\omega))$. In some settings, this latter model can be reduced to Problem~\eqref{eq:Decision-Dependent-Uncertainty}, as observed in \cite{munoz2024exploiting}.

As we mentioned above, most literature focuses on how the properties of the map $\beta:x\mapsto \beta_x$ aid to solve Problem~\eqref{eq:Decision-Dependent-Uncertainty}. In contrast, the focus of the Bayesian approach in \cite{salas2023existence} is rather on which kind of measure maps $\beta$ can be constructed as beliefs over a given moving set $S:X\tto Y$, such that the induced Problem~\eqref{eq:Decision-Dependent-Uncertainty} is solvable. That is, the data of the problem is not the measure map $\beta$ itself, but rather the moving support $S$.  

In \cite{salas2023existence} (as well as in \cite{Mallozzi1996}), the main contributions are on existence of solutions of Problem~\eqref{eq:Decision-Dependent-Uncertainty}, and so, on continuity properties of the mapping $$\phi:x\in X\mapsto \phi(x):=\E_{\beta_x}[\theta(x,\cdot)].$$ Particularly in \cite{salas2023existence}, it is observed that continuity of $\phi$ can be deduced for a large family of beliefs by studying how the neutral belief behaves, which is given by a uniform distribution over the moving set $S(x)$ (see equation~\eqref{eq:def neutral belief} for the formal definition). The main results of \cite{salas2023existence} in this sense are that $\phi$ is continuous for the neutral belief over $S$ in the following cases:
\vspace{.2cm}

\begin{enumerate}
    \item[(I)] the set-valued map $S:X\tto Y$ is continuous with full-dimensional, compact, convex values. 
    \item[(II)] the set-valued map $S:X\tto Y$ is given as the solution set of a bounded parametric fully linear problem, that is, $S(x) := \argmin_y\{ c^{\top}y\colon Ax+By\leq b \}$.    
\end{enumerate}
\vspace{.2cm}

In this work, we continue the study of the expected value function $\phi$, exploring sufficient conditions on the set-valued map $S$ to deduce Lipschitz continuity of $\phi$. Lipschitz continuity of $\phi$ can open the door to numerical treatments of Problem~\eqref{eq:Decision-Dependent-Uncertainty} and it is definitely a desired property in first-order analysis and optimization. In particular, subgradient methods are well-fitted for Lipschitz functions, when using the Clarke subdifferential or the Mordukhovich subdifferential, as they are nonempty at each point of the domain (see, e.g., \cite{rockafellar2009variational}). Also, in the context of constrained optimization if the objective is Lipschitz it allows us to use penalization techniques, see, e.g., \cite{Xiao2026Developing} and the references therein. Again in this setting, Lipschitz continuity of $S$ (in the sense of Hausdorff distances) and Lipschitz continuity of the integrand $\theta$ are not sufficient to guarantee Lipschitz continuity of the expected value function (see Example~\ref{ex:lip svm nonlip belief}, below). 

Our main results are that the map $\phi$ is Lipschitz continuous for the neutral belief over $S$ in the same cases (I) and (II) above, related to \cite{salas2023existence} (see Theorem~\ref{thm:neutral belief is radon lipschitz} and Proposition~\ref{prop:reduction to neutral beliefs} part 1 for (I), and Corollary \ref{cor:linear bilevel} for (II)).  
To the best of our knowledge, the properties of decision-dependent distributions induced by given moving sets and their applications to bilevel programming have only been studied in \cite{munoz2024exploiting,salas2023existence}, and the precursor works \cite{Mallozzi1996,MallozziMorgan2005}. Sufficient conditions to obtain Lipschitz continuity in this context have not yet been reported in the literature.

The rest of the paper is organized as follow. In Section~\ref{sec:pre}, we revise some preliminaries on Lipschitz analysis, set-valued maps and measure theory, and we present some initial results. In Section~\ref{sec:Convex}, we study some Lipschitz properties of maps defined over the space of convex compact sets, which are technical lemmas needed for the main results. In Section~\ref{sec:lipschitz beliefs}, we present our main results: sufficient conditions on $S$ to obtain Lipschitz continuity of $\phi$, when $\beta$ is the neutral belief. Section~\ref{sec:bilevel} is devoted to apply the results in the context of bilevel programming, where we show Lipschitz continuity of the expected value function in the context of approximate bilevel programming and fully linear bilevel programming. We finish the work with some conclusions and perspectives in Section~\ref{sec:conclusion}.


\section{Preliminaries and problem formulation} \label{sec:pre}
We consider the space $\R^m$ endowed with its usual inner product $\langle\cdot,\cdot\rangle$ and the induced norm $\|\cdot\|$. We denote by $\Ball_m$ and $\sph_{m-1}$ the unit ball and unit sphere in $\R^m$, respectively. 

Given $x\in\R^m$ and $r>0$ we write $B(x,r)$ for the open ball centered at $x$ with radius $r>0$, and $\overline{B}(x,r)$ for the corresponding closed ball. Given a nonempty set $A\subseteq \R^m$ we denote its diameter, interior, relative interior, affine hull, convex hull  and boundary by $\diam(A)$, $\inte(A)$, $\ri(A)$, $\mathrm{aff}(A)$, $\conv(A)$ and $\partial A$, respectively. We denote the distance of $x$ to $A$ as $d(x,A)$. If $A$ is closed and convex, we write $\proj(x,A)$ to denote the projection of $x$ to $A$. The dimension of a convex set $A\subset\R^m$, denoted by $\dim(A)$, is the dimension of its affine hull. We use the same notation (whenever it makes sense) for points and sets in arbitrary metric spaces.

\subsection{Preliminaries on metric analysis} Since we will work with several metric spaces afterward, we present the elements of metric analysis considering two arbitrary metric spaces, $(M,d_M)$ and $(N,d_N)$.

Let $f:M\to N$ be a function. We say that $f$ is Lipschitz (on $A\subseteq M$ resp.) if there exists a constant $L> 0$ such that
\[
    d_N(f(x),f(x'))\leq Ld_M(x,x')\qquad  \forall x,x'\in M \, (A\text{ resp.}).
\]
In that case, we call $L$ a Lipschitz constant for $f$ (on $A$, resp.) and we say that $f$ is $L$-Lipschitz (on $A$, resp.). We define $\Lip(f)$ (or sometimes $\Lip_{d_N}(f)$ to emphasize on the metric considered in $N$) as the infimum of such Lipschitz constants.\\    
Given $x\in M$, we say that $f$ is locally Lipschitz at $x$ if there exists $\delta>0$ such that $f$ is Lipschitz on $A=B(x,\delta)$. In that case, we define the Lipschitz number of $f$ at $x$ as
\begin{equation}\label{eq:localLipschitzConstant}
\Lip(f,x):=\inf_{\delta>0} \Lip\left(f|_{B(x,\delta)}\right),
\end{equation}
and, again, we may write $\Lip_{d_N}(f,x):=\Lip(f,x)$ to emphasize the dependence on the metric in $N$. We say that $f$ is uniformly locally Lipschitz if the constant $\Lip(f,x)$ is uniformly bounded for $x\in M$.

Finally, we say that $f$ is calm,  (see e.g. \cite[Section 8.F]{rockafellar2009variational}) at $x\in M$ if there exists $L> 0$ and $\delta>0$ such that
$$d_N(f(x),f(x'))\leq Ld_M(x,x')\qquad \forall x'\in B(x,\delta)$$
and, in that case, we define the modulus of calmness of $f$ at $x$,  $\calm(f,x)$, as the infimum of such constants $L$ as $\delta$ vanishes, or equivalently
\begin{equation}\label{eq:Punctually Lipschitz}
 \calm(f,x):=\limsup_{x'\to x} \frac{d_N(f(x),f(x'))}{d_M(x,x')}.
\end{equation}
We say that $f$ is uniformly calm if the constant $\calm(f,x)$ is uniformly bounded for $x\in M$.

It is well-known that the concepts of local Lipschitz continuity and (global) Lipschitz continuity are equivalent if the domain is a compact space (see, e.g., \cite[Theorem 2.1.6]{cobzacs2019lipschitz}). In contrast, uniform calmness is weaker than uniform local Lipschitz continuity, even in compact spaces. Indeed, we can take 
$$M:=\{0\}\cup \bigcup_{n\in\N}\underbrace{[n^{-1}-(2n)^{-2},n^{-1}+(2n)^{-2}]}_{=:I_n}$$ 
and $f:M\to\R$ given by $f(0):=0$ and $f(x):=n^{-2}\sin(n)$ for $x\in I_n$ with $n\in \N$. Clearly, $M$ is compact and $\calm(f,x)=0$ for all $x\in M$ while $f$ is not locally Lipschitz around $0$.
However, there is a reasonable framework where these concepts coincide, namely, the \emph{quasiconvex spaces}. A metric space $(M,d_M)$ is said to be $c$-quasiconvex with $c\geq 1$ (see e.g. \cite{LiuZhou2023HJmetricspaces} and the references therein) if for every $x,y\in M$ there exists a continuous curve $\gamma:[0,1]\to M$ connecting $x$ and $y$ (i.e. $\gamma(0) = x$ and $\gamma(1) = y$), such that $\ell(\gamma) \leq cd_M(x,y)$, where $\ell(\gamma)$ is the length of $\gamma$, that is,
\begin{equation}\label{eq: length}
    \ell(\gamma) := \sup\left\{ \sum_{i=0}^n d(\gamma(t_i),\gamma(t_{i+1})) \right\},
\end{equation}
where the supremum is taken over all partitions $0=t_0<t_1<\cdots<t_{n+1}=1$ of the interval $[0,1]$ (see, e.g. \cite{LiuZhou2023HJmetricspaces,AmbrosioGigliSavare2008}). If $M$ is $c$-quasiconvex for every $c>1$, then $M$ is said to be a \textit{length space}. Finally, if $M$ is 1-quasiconvex, then it is a \textit{geodesic space}. Quasiconvex spaces include convex sets and compact manifolds, among their most notable examples.

\begin{lemma} \label{lem:punctual to local lip}
    Let $M$ be a $c$-quasiconvex space and $f:M\to N$ a function. 
    \begin{enumerate}
    \item Assume that $f$ is uniformly calm, that is, there exists $L>0$ such that $\calm(f,x)\leq L$ for all $x\in M$. Then $f$ is Lipschitz and $\Lip(f)\leq cL$.
    \item Let $\bar{x}\in M$ and assume that $f$ is uniformly calm near $\bar{x}$, that is, there exists $L>0$ and $\delta>0$ such that $\calm(f,x)\leq L$ for all $x\in B(\bar x,\delta)$. Then $f$ is locally Lipschitz around $\bar{x}$ and $\Lip(f,\bar{x})\leq cL$. 
    \end{enumerate}
\end{lemma}

\begin{proof}
    Let us first show the first statement. Consider $x_0,x_1\in M$. Set $\varepsilon>0$ and consider a continuous curve $\gamma:[0,1]\to M$ from $\gamma(0)=x_0$ to $\gamma(1)=x_1$ with $\ell(\gamma)\leq cd_M(x_0,x_1)$. For each $t\in [0,1]$ we have that $\calm(f,\gamma(t))<L+\varepsilon$, and so there exists $\delta_t>0$ such that for all $x\in B(\gamma(t),\delta_t)$
    \[d_N(f(\gamma(t)),f(x))\leq (L+\varepsilon)d_M(\gamma(t),x).
    \]
    By the continuity of $\gamma$ for each $t$ we can find $\rho_t>0$ such that
    \[d_N(f(\gamma(t)),f(\gamma(s)))\leq (L+\varepsilon)d_M(\gamma(t),\gamma(s))    \]
    whenever $|s-t|<\rho_t$. Let us define the set 
    \[
    A=\left\{a\in[0,1]:\begin{array}{c}
        \exists k\in\N, \; \exists (t_i)_{i=0}^k \text{ with } t_0=0,\; t_k=a,\text{ such that }\\ 0<t_i-t_{i-1}<\max(\rho_{t_i},\rho_{t_{i-1}}) \quad \forall i\in[k]
    \end{array}\right\}
    \]

    We claim that $1\in A$. Under this claim there exists $0=t_0<t_1<\ldots<t_k=1$ such that $|t_i-t_{i-1}|$ is either less that $\rho_{t_i}$ or less than $\rho_{t_{i-1}}$ for all $i\in[k]$. Therefore we have 
$$\begin{aligned}
d_N(f(x_0),f(x_1))&\leq \sum_{i\in{k}} d_N(f(\gamma(t_{i-1})),f(\gamma(t_i)))\\
&\leq \sum_{i\in[k]}(L+\varepsilon)d_M(\gamma(t_{i-1}),\gamma(t_i))\\
&\leq (L+\varepsilon)\ell(\gamma)\leq (L+\varepsilon) cd_M(x_0,x_1).
\end{aligned}$$
Since this is true for any $\varepsilon>0$ we deduce that $f$ es $cL$ Lipschitz.

It only remains to prove the claim $1\in A$. First we will show that $A$ has a maximum. Indeed, if $m=\sup A$ then there exists $a\in A$ with $a> m-\frac{\rho_m}{2}$, and so also $(t_i)_{i=0}^k$ such that $0=t_0<t_1<\ldots<t_k=a$ with
$$t_i-t_{i-1}< \max(\rho_{t_i},\rho_{t_{i-1}}).$$
Then defining $t_{k+1}:=m$ we see that $m\in A$ as  
$$t_{k+1}-t_k=m-a<\rho_m/2\leq \max(\rho_{t_{k+1}},\rho_{t_k})$$
showing that $m$ is the maximum of $A$. To end, let us assume  by contradiction that $m<1$. In that case considering $t_{k+1}=\min(1,t_k+\frac{\rho_{t_k}}{2})$ clearly shows that $t_k<t_{k+1}\in A$ which is a contradiction.

   For the second part, we localize the argument above. Take $\bar{x}\in M$ and $\delta,L>0$ such that $\calm(f,x)\leq L$ for all $x\in B(\bar{x},\delta)$. Set $\eta = \delta/2c$, and
    take $x_0,x_1\in B(\bar{x},\eta)$. Set $\varepsilon>0$ and consider a continuous $\gamma:[0,1]\to M$ from $\gamma(0)=x_0$ to $\gamma(1)=x_1$ with $\ell(\gamma)\leq cd_M(x_0,x_1)$. Then, 
    \[
    \gamma([0,1])\subset B(x_0,\delta/2)\subset B(\bar{x},\delta/2+\eta)\subset B(\bar{x},\delta),
    \]
    and thus, for each $t\in [0,1]$ we have that $\calm(f,\gamma(t))<L+\varepsilon$. The rest of the proof follows as above, deducing that $f$ is $cL$-Lipschitz in $B(\bar{x},\eta)$.
\end{proof}

The following lemma summarizes useful calculus rules for locally Lipschitz functions and can be found for instance in \cite{cobzacs2019lipschitz}, namely Propositions 2.3.1, 2.3.2, 2.3.3, 2.3.4 and 2.3.7.
\begin{lemma}
\label{lem:lipschitz algebra}
Let $u,v:M\to\R^m$ and $\alpha:M\to\R$ be locally Lipschitz functions around $\bar{x}\in M$, and $\varphi:A\subset\R^m\to\R^m$ a locally Lipschitz function around $u(\bar{x})\in A$. Then, the following functions are locally Lipschitz around $\bar{x}$:
\begin{enumerate}\setlength{\itemsep}{0.3cm}
    \item $\|u\|$, with $\Lip(\|u\|,\bar{x}) \leq \textup{Lip}(u,\bar{x})$;
    \item $\alpha u+v$, with $\Lip(\alpha u+v,\bar{x})\leq |\alpha(\bar{x})|\Lip(u,\bar{x})+\|u(\bar{x})\|\Lip(\alpha,\bar{x})+\Lip(v,\bar{x})$;
    \item $\langle u,v\rangle$, with $\Lip(\langle u,v\rangle,\bar{x})\leq \|u(\bar{x})\|\Lip(u,\bar{x})+\|v(\bar{x})\|\Lip(v,\bar{x})$;
    \item $\varphi\circ u$, with $\Lip(\varphi\circ u,\bar{x}) \leq \Lip(\varphi,u(\bar{x}))\Lip(u,\bar{x})$;
    \item $\frac{1}{\alpha}$, with $\Lip\left(\frac{1}{\alpha},\bar{x}\right)\leq \frac{1}{(\alpha(\bar{x}))^{2}}\Lip(\alpha,\bar{x})$, provided $\alpha(\bar{x})\neq 0$;
    \item $\frac{u}{\|u\|}$, with $\Lip\left(\frac{u}{\|u\|},\bar{x}\right)\leq \frac{1}{\|u(\bar{x})\|}\Lip(u,\bar{x})$, provided $u(\bar{x})\neq 0$.
\end{enumerate}
\end{lemma}

\subsection{General notation} Let $Y$ be a nomempty compact subset of $\R^m$ and $X$ a metric space. Recall that a set-valued map $S:X\tto Y$ is a function that assigns to each $x\in X$ a subset $S(x)$ of $Y$, which we refer to as the image of $x$. 
For two nonempty closed subsets $A,B\subseteq Y$ (hence compact), we write
\begin{equation}\label{eq:HausdorffDistance}
   d_H(A,B):=\max(e(A,B),e(B,A)),
\end{equation}
where $e(A,B):=\sup \{d(a,B): a\in A\}$ is the excess of $A$ over $B$. If the set-valued map $S$ has closed values, we say that $S$ is Lipschitz if it is so for the Hausdorff metric $d_H$. That is, if there exists $L>0$ such that
\begin{equation}\label{eq:Lipschitz-Set-Valued}
    d_H(S(x),S(x'))\leq Ld(x,x'),\qquad \forall x,x'\in X.
\end{equation}

We write $\lambda$ and  $\haus_{k}$ to denote the Lebesgue measure and the $k$-dimensional Hausdorff measure in $\R^m$, respectively. We write $\mathcal{B}(Y)$ and $\P(Y)$ to denote the Borel $\sigma$-algebra and the space of Borel probability measures over $Y$, respectively. In $\P(Y)$, we consider the following distance functions (see e.g. \cite[p. 385]{klenke2013probability}):
\begin{itemize}
    \item The \emph{Total Variation distance} $\dTV:\P(Y)\times \P(Y)\to\R_+$, defined by 
    \begin{equation}\label{eq:RadonDistance-def}
    \dTV(\mu,\nu):=\sup\{\E_\mu[f]-\E_\nu[f]\mid f:Y\to\R\text{ measurable}, \|f\|_{\infty}\leq 1\}.
    \end{equation}
    \item The \emph{Wasserstein-1 distance} $\dW:\P(Y)\times \P(Y)\to\R_+$, defined by 
    \begin{equation}
    \dW(\mu,\nu):=\sup\{\E_\mu[f]-\E_\nu[f]\mid f:Y\to\R \text{ Lipschitz}, \Lip(f)\leq 1\}.
    \end{equation}
\end{itemize}

Note that since we assume $Y\subseteq \R^m$ is compact, then $\dW\leq \frac{1}{2}\diam(Y)\dTV$. Recall from \cite{salas2023existence} that a map $\beta:X\to \P(Y)$ is said to be a \emph{belief} over a set-valued map $S:X\tto Y$ if for each $x\in X$, $\beta_x:=\beta(x)$ concentrates on $S(x)$, that is, if
\begin{equation}\label{eq:Def-Belief}
    \beta_x(S(x)) = 1,\qquad\forall x\in X.
\end{equation}

\subsection{Problem formulation}\label{subsec:formulation}
Let $X$ be a metric space, $Y$ a nonempty compact subset of $\R^m$, $S:X\tto Y$ a set-valued map with nonempty closed images, and let $\beta:X\to\P(Y)$ be a belief over $S$.

We consider the following model. A decision maker chooses $x\in X$, and after that a random variable $y$ is drawn following the decision-dependent distribution $\beta_x$ whose support is $S(x)$. The distribution $\beta_x$ models how the decision maker \emph{believes} the actual realization $y\in S(x)$ is selected by nature, and considers this information to decide $x\in X$. The cost of the decision $x$ with the realization $y$ of the random parameter is $\theta(x,y)$ where $\theta:X\times \R^m\to \R$ is a given cost function. Then the problem of the decision maker, central to this paper, can be posed as the following 
stochastic decision-dependent optimization problem
\begin{equation}
    \min_{x\in X} \E_{\beta_x}[\theta(x,\cdot)].
\end{equation} 

Our aim is to study the Lipschitz continuity of the objective function $\phi:x\in X\mapsto \E_{\beta_x}[\theta(x,\cdot)]$ in the setting where $\beta$ is a belief over the set-valued map $S$. The following proposition presents a natural sufficient condition: the Lipschitz continuity of $\phi$ can be obtained by studying the Lipschitz continuity of $\beta$ and $\theta$.
\begin{proposition}
    \label{prop:reduction to neutral beliefs} 
    Let $(X,d)$ be a compact metric space, and $Y$ be a nonempty compact convex subset of $\R^m$. Let $\beta:X\to\P(Y)$ be a belief and $\theta:X\times Y\to \R$ a function. Assume that at least one of the following holds:
    \begin{enumerate}
        \item $\beta$ is Lipschitz with respect to $\dTV$ and $\theta$ is continuous and uniformly Lipschitz in the first variable.
        \item $\beta$ is Lipschitz with respect to $\dW$ and $\theta$ is Lipschitz.
    \end{enumerate} Then $x\mapsto \phi(x):=\E_{\beta_x}[\theta(x,\cdot)]$ is Lipschitz.
\end{proposition}

\begin{proof}
    In the first case we have that there exists $L>0$ such that $\theta(\cdot ,y )$ is $L$ Lipschitz for all $y\in Y$. Then we have that,
\[
\begin{aligned}
    |\phi(x)-\phi(x')|&\leq |\E_{\beta_x}[\theta(x,\cdot )]-\E_{\beta_{x'}}[\theta(x,\cdot )]|+|\E_{\beta_{x'}}[\theta(x,\cdot )]-\E_{\beta_{x'}}[\theta(x',\cdot )]|\\
    &\leq \|\theta(x,\cdot )\|_\infty \Lip_{\mathrm{TV}}(\beta)d(x,x')+\E_{\beta_{x'}}[|\theta(x,\cdot)-\theta(x',\cdot)|]\\
    &\leq (\|\theta\|_\infty \Lip_{\mathrm{TV}}(\beta)+L)d(x,x').
\end{aligned}    
    \]
    Therefore, $\phi$ is $(\|\theta\|_\infty \Lip_{\mathrm{TV}}(\beta)+ L)$ Lipschitz.

    For the second case let $L$ be a Lipschitz constant for $\theta$. Then 
    \[
\begin{aligned}
|\phi(x)-\phi(x')|&\leq |\E_{\beta_x}[\theta(x,\cdot )]-\E_{\beta_{x'}}[\theta(x,\cdot )]|+|\E_{\beta_{x'}}[\theta(x,\cdot )]-\E_{\beta_{x'}}[\theta(x',\cdot )]|\\
&\leq L\cdot \Lip_{W_1}(\beta)d(x,x')+L d(x,x'),\\
\end{aligned}
\]
so that $\phi$ is $L(\Lip_{W_1}(\beta)+1)$ Lipschitz. 
\end{proof}

In the sequel, we pay particular attention to the \emph{neutral belief} over $S$ defined as $\iota:X\to\P(Y)$, where for each $x\in X$ and $A\in\mathcal{B}(Y)$
\begin{equation}
    \label{eq:def neutral belief}
    \iota_x(A):=\frac{\lambda_x(A\cap S(x))}{\lambda_x(S(x))},
\end{equation} 
with $\lambda_x$ denoting the Lebesgue measure over the affine space generated by $S(x)$.
We will say that a belief $\beta$ over $S$ has \emph{density} $h$ if, $h$ is a strictly positive function over $X\times Y$ and for any $A\in \mathcal{B}(Y)$ we have
\begin{equation}
\label{eq:def belief with density}
    \beta_x(A):=\frac{\int_{A\cap S(x)}h(x,y)d\lambda_x(y)}{\int_{S(x)}h(x,y)d \lambda_x(y)}.
\end{equation}

\begin{remark}
    A natural example of these beliefs are those constructed from densities of the form 
    $h(x,y) = \exp(-\psi(x,y))$, where $\psi$ is Lipschitz continuous on $x$ and uniformly Lipschitz continuous on $y$. This family induces, 
     as a particular case, the beliefs given by \textit{truncated normal distributions} over $S(x)$. Indeed, it is enough to take 
    \[
        \psi(x,y) = \frac{1}{2}(y-\mu(x))^\top \Sigma(x)^{-1}(y-\mu(x)),
    \]
    provided $\mu(x)$ is Lipschitz and $\Sigma(x)$ is Lipschitz and uniformly positive 
    definite, i.e., $\Sigma(x) \geq \sigma I$ for some $\sigma > 0$ independent of $x$. Moreover, similar analysis can be done considering any classic parametric density $f(y;p)$ (like Exponential distributions of the form $f(y;\lambda):=1_{\{y\geq 0\}}\lambda e^{-\lambda y}$ or Gamma distributions $f(y;\alpha,\beta):= \Gamma(\alpha)^{-1}\beta^{\alpha}y^{\alpha-1}e^{-\beta y}$ for $y>0$), carefully restricting the ambient domain $Y$, and taking $p(x)$ as a regular enough varying parameter. The exact hypotheses on $Y$ and $p$ should be tuned so that $h(x,y):= f(y;p(x))$ remains Lipschitz on $x$ and uniformly Lipschitz on $y$.
\end{remark}

The following is a corollary of Proposition \ref{prop:reduction to neutral beliefs} showing that the analysis of beliefs with densities over $S$ can somehow be reduced to the neutral belief over $S$.
\begin{corollary}\label{cor:Reduction-to-neutral}
    Let $\iota$ be the neutral belief over $S$ and $\beta$ be another belief over $S$ with a density $h$.
    \begin{enumerate}
        \item If $\iota$ is Lipschitz with respect to $\dTV$ and $h$ is continuous and uniformly Lipschitz in the first variable, then $\beta$ is also Lipschitz with respect to $\dTV$.
        \item If $\iota$ is Lipschitz with respect to $\dW$ and $h$ is Lipschitz, then $\beta$ is also Lipschitz with respect to $\dW$.
    \end{enumerate}
\end{corollary}
\begin{proof}
    Let $f:Y\to \R$ be measurable with $\|f\|_\infty \leq 1$ in the first case, and Lipschitz with $ \Lip (f)\leq 1$ in the second one. We then can write 
    \begin{equation}
    \label{eq:expected fraction}
        \begin{aligned}
            \E_{\beta_x}[f]=\frac{\E_{\iota_x}[h(x,\cdot)f(\cdot)]}{\E_{\iota_x}[h(x,\cdot)]}.
        \end{aligned}
    \end{equation}
    It follows from Proposition \ref{prop:reduction to neutral beliefs} that both the numerator and the denominator in \eqref{eq:expected fraction} are Lipschitz and hence by Lemma \ref{lem:lipschitz algebra} the map $x\mapsto \E_{\beta_x}[f]$ is Lipschitz, as $h$ and  $\E_{\iota_x}[h(x,\cdot)]$ are positive and bounded away from zero. The Lipschitz constant can be shown to be uniform over $f$, therefore yielding that $\beta$ is Lipschitz with respect to $\dTV$ in the first case and Lipschitz with respect to $\dW$ in the second.
\end{proof}

We end this section with an example showing that for the neutral belief, Lipschitz data is not sufficient to guarantee that the expected value is Lipschitz, and hence neither for the belief (due to Proposition \ref{prop:reduction to neutral beliefs}).

\begin{example}
\label{ex:lip svm nonlip belief}
    Let $S:[0,1]\tto [0,1]^2$ be given by 
    \begin{equation}
    \label{eq:example a(x)}
    S(x):=\conv\{(0,0),(1,0),(1,x),(a(x),x)\},\quad x\in [0,1]
    \end{equation}
    with $a(x):=\sqrt[4]{x}$. We observe that $S$ is 1-Lipschitz. Indeed, for $x'<x$ we have $S(x')\subseteq S(x)$ and so 
    \[
    \begin{aligned}
    d_H(S(x),S(x'))&=e(S(x),S(x'))=d((a(x),x),S(x'))=|x-x'|.
    \end{aligned}
    \]
    However, we shall see that $\iota$ the neutral belief over $S$ is not Lipschitz with respect to $\dW$. Indeed, consider the function $\theta(x,y)=y_1$, where $y_1$ is the first coordinate of $y=(y_1,y_2)\in\R^2$. It is clear that $\theta$ satisfies $\Lip(\theta)\leq 1$. We have that the volume of the trapezoid $S(x)$ is $\lambda(S(x))=(1-\sqrt[4]{x}/2)x$ and so
$$
\begin{aligned}    
\phi(x)=\E_{\iota_x}[\theta(x,\cdot)]&=\frac{1}{\lambda(S(x))}\int_{S(x)}y_1dy\\
&=\frac{2}{x(2-\sqrt[4]{x})}\int_0^x\int_{\frac{x}{\sqrt[4]{x}}y_2}^1y_1 dy_2dy_1 =\frac{3-\sqrt{x}}{6-3\sqrt[4]{x}}.
\end{aligned}
$$
We can compute the derivative of $\phi$ for $x>0$ as follows
    $$
\begin{aligned}
\phi'(x)&=\frac{(-\frac{1}{2}x^{-1/2})(6-3x^{1/4})-(3-x^{1/2})(-\frac{3}{4}x^{-3/4})}{(6-3\sqrt[4]{x})^2}\\
        &=\underbrace{\frac{1}{x^{\frac{3}{4}}}}_{\to\infty}\cdot\underbrace{\frac{-3x^{1/4}+\frac{3}{4}x^{1/2}+\frac{9}{4}}{(6-3\sqrt[4]{x})^2}}_{\to \frac{1}{16}}\to \infty \qquad \text{ as } x\to 0^+.
    \end{aligned}$$
    Since $\phi'$ is not bounded then $\phi$ cannot be Lipschitz. Moreover, we deduce using Proposition \ref{prop:reduction to neutral beliefs} part 2 that $\iota$ is not Lipschitz with respect to $\dW$. 
\hfill$\Diamond$
\end{example}

\begin{remark}
The set-valued map $S$ of \eqref{eq:example a(x)} is not only Lipschitz, but also \emph{rectangularly continuous}, which is an extra sufficient condition to deduce continuity of the neutral belief in \cite{salas2023existence}. Thus, $\phi$ is continuous but not Lipschitz. We also observe that for the chosen function $\theta$ the expected value corresponds to the first coordinate of the centroid of the trapezoid $S(x)$. In Figure \ref{fig:non lipschitz centroid of trapezoids}, some images of $S$ (trapezoids) are depicted along with their centroids $c(S(x))$, which can be computed as $$c(S(x))=\left(\frac{3-\sqrt{x}}{6-3\sqrt[4]{x}},  
\frac{x(3-2\sqrt[4]{x})}{
6-3\sqrt[4]{x}}
\right).$$ Figure \ref{fig:non lipschitz centroid of trapezoids} depicts the non-Lipschitzian property of the centroids.
\end{remark}

\begin{figure}[h]
\centering
\begin{minipage}{0.54\textwidth}
    \tdplotsetmaincoords{80}{140}
			\begin{tikzpicture}[scale=3.5,,>=latex,
				tdplot_main_coords,
				axis/.style={-stealth,very thick}
				]
				\draw[thick,dashed] (0,-1,0)--(-1,-1,0)--(-1,0,0);
				\draw[thick] (-1,0,0)--(0,0,0)-- (0,-1,0) -- (-1,-1,1)--(-1,0,0);
				\draw[dashed,thick] (-1,-1,1)--(-1,-1,0);
				\draw[domain=0:1,samples=500,smooth,variable=\x,thick] plot ({-\x},-{\x^(1/4)},{\x});

                \fill[fill = yellow, fill opacity=0.3] (-1,-1,0)--(-1,0,0)--(-1,-1,1)--cycle;

				\draw[fill = green, fill opacity=0.3] (-0.6561,-1,0)--(-0.6561,-1,0.6561)--(-0.6561,-0.9,0.6561)--(-0.6561,0,0) --cycle;

                \draw[fill = cyan, fill opacity=0.3] (-0.2401,-1,0)--(-0.2401,-1,0.2401)--(-0.2401,-0.7,0.2401)--(-0.2401,0,0) --cycle;


                \def\ox{0.8};
                \def\oz{1};
				\draw[gray,->] (\ox,0,\oz) -- ({\ox-0.3},0,\oz) node[black,right]  {$x$};
				\draw[gray,->] (\ox,0,\oz) -- (\ox,-0.3,\oz) node[black,left]  {$y_1$};
				\draw[gray,->] (\ox,0,\oz) -- (\ox,0,\oz+0.3) node[black,above]  {$y_2$};	

                \foreach \t in {0.3,0.35,...,1}
                \draw[gray, opacity=0.3] (0,-\t,0) -- ({-\t^4},{-\t},{\t^4});
                
                \foreach \t in {0.05,0.1,...,0.95,0.96}
                \draw[domain=0:1,samples=500,smooth,variable=\x,thick,gray,opacity=0.3] plot ({-\x},-{\t*\x^(1/4)},{\t*\x});

                \foreach \x in {0.01,0.04,...,0.99}
                \draw[gray, opacity = 0.3] ({-\x},0,0)--({-\x},{-(\x^(1/4))},{\x})--({-\x},-1,{\x});
                
                \foreach \x\y in {1/yellow,.9/green,.7/cyan}
                \filldraw[fill=\y] ({-\x^4},{-(3-\x^2)/(6-3*\x)},{\x^4*(3-2*\x)/(6-3*\x)}) circle (.5pt); 
                \draw[thick, samples=500, domain=0:1, smooth, variable=\x, black] plot ({-\x^4},{-(3-\x^2)/(6-3*\x)},{\x^4*(3-2*\x)/(6-3*\x)}); 
                \fill[black] (0,-1/2,0) circle (.5pt);
			\end{tikzpicture}
\end{minipage}
\begin{minipage}{0,45\textwidth}
\begin{tikzpicture}[scale=4,>=latex]
    \draw[gray,->] (0,0) -- (1.1,0) node[right,black] {$y_1$};
    \draw[gray,->] (0,0) -- (0,1) node[above,black] {$y_2$};
    
\foreach \x\y in {1/yellow,.9/green,.7/cyan}
    \filldraw[thick,gray!30!white, fill=\y!30!white] (0,0) -- (1,0) -- (1,{\x^4}) -- (\x,{\x^4}) -- cycle;

\foreach \x\y in {1/yellow,.9/green,.7/cyan}
    \filldraw[fill=\y] ({(3-\x^2)/(6-3*\x)},{\x^4*(3-2*\x)/(6-3*\x)}) circle (.5pt); 

    \fill[black] (1/2,0) circle (.5pt);

\draw[thick, scale=1, domain=0:1, smooth, variable=\x, black] plot ({(3-\x^2)/(6-3*\x)},{\x^4*(3-2*\x)/(6-3*\x)}); 
\end{tikzpicture}
\end{minipage}
\caption{Overlapped values $S(1)$, $S(0.9^4)$ and $S(0.7^4)$ of the set-valued map $S$ of Example \ref{ex:lip svm nonlip belief} and their centroids, depicted with circles of the same color, exhibiting a non-Lipschitz behavior as $x\to 0$.}
\label{fig:non lipschitz centroid of trapezoids}
\end{figure}
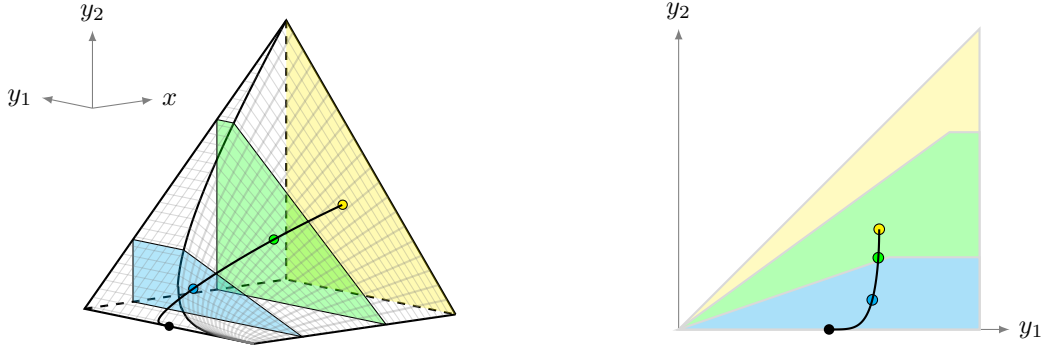


\section{Some properties on the space of compact convex sets}\label{sec:Convex}

To derive the main results present in Section \ref{sec:lipschitz beliefs}, we need to study how the Lebesgue measure behaves over the family of nonempty compact convex sets. Recall that $Y\subset \R^m$ is a nonempty convex compact set. Let us denote
\begin{equation}\label{eq:Def-SpaceOfConvexSets}
    \D_Y := \{ K\subset Y\,:\, K\text{ is nonempty, convex and compact} \}.
\end{equation}
We observe that, $\D_Y$ endowed with the Hausdorff $d_H$ is a compact geodesic metric space (as a consequence of \cite[Theorem 4.18]{rockafellar2009variational}) and \cite[Proposition 1]{serra1998hausdorff}). 

\subsection{Some Lipschitz maps on the space of convex sets} In this section, we revise some maps related with the space $\mathcal{D}_Y$ verifying Lipschitz continuity. 
\begin{lemma}\label{lem:diam lip}
    The function $\diam:\D_Y\to \R$ given by $\diam(A):=\sup\{ \|x-y\| : x,y\in A \}$ is 2-Lipschitz.
\end{lemma}
\begin{proof}
    Let $A,B\in \D_Y$. For $\varepsilon>0$, let $x,y\in A$ such that $\diam(A)\leq \|x-y\|+\varepsilon$. We see that
    \[ 
    \begin{aligned}
        \diam(A)&\leq \|x-y\|+\varepsilon \\
        & \leq \|x-\proj(x,B)\| + \|\proj(x,B)-\proj(y,B)\| + \|\proj(y,B)-y\| +\varepsilon \\
        & \leq 2 d_{H}(A,B) + \diam(B) + \varepsilon.
    \end{aligned}
    \]
    Since last inequality is valid for any $\varepsilon>0$, we deduce that $\diam(\cdot)$ is 2-Lipschitz.
\end{proof}

\begin{lemma}\label{lem:lipvolume}
    The volume function (Lebesgue measure) $\lambda$ restricted to $\D_Y$, which for every $A\in \D_Y$ assigns the full-dimensional Lebesgue measure of $A$, $\lambda(A)$, is Lipschitz with respect to $d_H$. The Lipschitz constant satisfies $\Lip(\lambda)\leq L_{Y,m}$ where
    \begin{equation}\label{eq: LipConstant-Lebesgue}
    L_{Y,m}:=  2m\lambda(\Ball_m)\left(\diam(Y)\sqrt{\frac{m}{2(m+1)}}\right)^{m-1}.
    \end{equation}
    Moreover, for any $C\in \D_Y$ and $L>L_{Y,m}$ there exists $\delta>0$ such that 
    \[
    \lambda(D\Delta E)\leq L d_H(D,E)\quad \forall D,E\in B_{\D_Y}(C,\delta).
    \]
\end{lemma}

\begin{proof}
    Without loss of generality, we can assume that $Y$ has nonempty interior. Let $C\in \D_Y$ and $\delta>0$, and take $D,E\in B_{\D_Y}(C,\delta)$. We note that
    \begin{equation}\label{eq:diference of values}
\begin{aligned}
    |\lambda(D)-\lambda(E)|&\leq \lambda(D\Delta E)=\lambda(D\setminus E)+\lambda(E\setminus D).
        \end{aligned}
    \end{equation}
    Let us bound the first term on the right hand side of \eqref{eq:diference of values}.
    Let $\varepsilon:=d_H(D,E)$. Then, we have $D\subseteq E+\varepsilon\Ball_m$ and so 
    \begin{equation}
        \label{eq:inclusionofdifferences}
   D\setminus E\subseteq (E+\varepsilon\Ball_m)\setminus E.
    \end{equation}
    By Jung's Theorem \cite{jung1901ueber}, we know that there exists a ball of radius \begin{equation}\label{eq:Jung inequality}
        r\leq \diam(E)\sqrt{\frac{m}{2(m+1)}}
    \end{equation} that encloses $E$, so for some point $z\in \R^m$ we have $E\subseteq B(z,r)$. Moreover, since $Y$ is compact, we have that $\diam(E)\leq\diam(Y)<+\infty$, and hence $r\leq \diam(Y)\sqrt{\frac{m}{2(m+1)}}$. Given any $t\in[0,\varepsilon]$ we have $E+t\Ball_m\subseteq \overline{B}(z,r+t)$ and by virtue of the monotonicity of perimeters of compact convex sets (see, e.g.  \cite[Lemma 2.4]{buttazzo1995minimum}) we have that
    \begin{equation}\label{eq:perimeter inequality}
        \haus_{m-1}(\partial (E+t\Ball_m)) \leq  \haus_{m-1}(\partial B(z,r+t)).
    \end{equation}
    Using the coarea formula (see e.g. \cite[Theorem 3.10]{evans2015measure}) we deduce that
    $$\begin{aligned}
        \lambda(D\setminus E)\leq \lambda(E+\varepsilon\Ball_m\setminus E)
        &=\int_0^\varepsilon \haus_{m-1}(\partial (E+t\Ball_m))dt \\
        &\leq \int_0^\varepsilon \haus_{m-1}(\partial (B(z,r+t))dt \\ 
        &=\lambda(B(z,r+\varepsilon)\setminus B(x,r))\\
        &= \haus(\Ball_m) [ (r+\varepsilon)^m-r^m] = \haus(\Ball_m)p(r,\varepsilon)\varepsilon,
    \end{aligned}$$
    where
    \( p(r,\varepsilon) = \sum_{i=1}^{m} \binom{m}{i} \varepsilon^{i-1} r^{m-i}.  \)
    Notice that since $0<r\leq \diam(Y)\sqrt{\frac{m}{2(m+1)}}$ and $0<\varepsilon=d_H(D,E)\leq 2\delta$, $p(r,\varepsilon)$ can be bounded by a polynomial $q$ on $\delta$, with $$q(0)=m\left(\diam(Y)\sqrt{\frac{m}{2(m+1)}}\right)^{m-1}.$$
    The second term in \eqref{eq:diference of values} can be bounded in an analogue manner. Putting all together and recalling that $\varepsilon=d_H(D,E)\leq 2\delta$ we have
    \begin{equation*}
        |\lambda(D)-\lambda(E)|\leq \lambda(D\Delta E)\leq 2\lambda(\Ball_m)q(\delta)d_H(D,E),
    \end{equation*}
    so that $2\lambda(\Ball_m)q(\delta)$ is a local Lipschitz constant for $\lambda$. Moreover, taking limit $\delta\to 0$ we deduce that the local Lipschitz number of $v$ at $C$ satisfies
    \[
    \Lip(\lambda,C)\leq 2\haus(\Ball_m)m\left(\diam(Y)\sqrt{\frac{m}{2(m+1)}}\right)^{m-1}.
    \]
    Since $(\D_Y,d_H)$ is a geodesic space, the conclusion follows by Lemma \ref{lem:punctual to local lip}.
\end{proof}


Given $u\in \sph_{m-1}$ and $K\in \D_Y$, we define $ r_K(u):=\sup\{t\geq 0: tu\in K\}$. The functional $r_K$ coincides with  the reciprocal of the Minkowski functional of $K$, restricted to $\sph_{m-1}$. 
We observe that if $0\in K$ then
$0\leq r_K(u)\leq \diam(Y)<+\infty$. 
We also define the \textbf{inner radius} of $K$ as
\begin{equation}
    r(K) := \inf\{r_K(u)\colon u\in \sph_{m-1}\cap \spn(K-K)\},
    \label{eq:inner radius}
\end{equation}
which, assuming $0\in \ri(K)$, satisfies 
$0<r(K)\leq\diam(Y)<\infty.$ Also, note that if $0\in \ri(K)$, then
\begin{equation}\label{eq:internal radius ref}
    r(K) = \max\{ r\colon \overline{B}(0,r)\cap \spn(K)\subset K\} = d(0,K\setminus\ri(K)).
\end{equation}

\begin{lemma}
\label{lem:inner radius lip}
Let $K\in \D_Y$ and assume $0\in\inte(K)$. Then the function $A\in \D_Y\mapsto r_A(u)$ is locally Lipschitz around $K$ uniformly in $u\in \sph_{m-1}$, that is, there exists $\delta, L>0$ such that $A,B\in \D_Y$ with $d_H(A,K),d_H(B,K)\leq \delta$ implies
\[
|r_A(u)-r_B(u)|\leq Ld_H(A,B)\qquad \forall u\in \sph_{m-1}.
\]
\end{lemma}
\begin{proof}
    We first claim that  $r:\mathcal{D}_{Y}\to \R$ which assigns to $A\in \D_Y$ the inner radius $r(A)$ is bounded below away from $0$ in some neighborhood of $K$. We note that since $0\in\inte(K)=\ri(K)$ then using \eqref{eq:internal radius ref} we have $\overline{B}(0,r(K))\subset K.$

    Indeed, take $\delta:=\frac{r(K)}{2}>0$ and consider $A\in \D_Y$ such that $d_H(K,A)\leq \delta$. We will show that $r(A)\geq \delta$. Using the definition of $r(A)$ we may take $x\in \partial A$ such that $\|x\|=r(A)$. We observe that $0\in \inte(A)$. Otherwise, by a separation argument there exists $\xi\in\R^m$ with $\|\xi\|=1$ such that
    \begin{equation}
    \label{eq: VI xi}    
    \langle \xi,z\rangle\leq 0 \quad \forall z\in A.
    \end{equation}
    We define $w:=\xi \cdot r(K)$ which satisfies $\|w\|=r(K)$ and so also $w\in K$. Clearly, from \eqref{eq: VI xi} the projection of $w$ on $A$ is 0, and so we obtain $$r(K)=\|w-0\|=d(w,A)\leq d_H(K,A)\leq \delta=\frac{r(K)}{2}$$ which is a contradiction, since $r(K)>0$. 
    
    Therefore, since $\overline{B}(0,r(A))\subset A$ then the normal cones (in the sense of convex analysis, see, e.g., \cite{rockafellar2009variational}) to these sets satisfy 
    \begin{equation}
    \label{eq:normal cones in proof}
    N_A(x)\subset N_{\overline{B}(0,r(A))}(x)=\R_+ x.
    \end{equation}
    Since $x\in\partial A$ and $A$ is convex in finite dimension, then the normal cone $N_A(x)$ is nontrivial, that is, it contains nonzero directions. By \eqref{eq:normal cones in proof}, we deduce that $N_A(x) = \R_+ x$.
    Now consider $z:=x\frac{r(K)}{r(A)}\in \overline{B}(0,r(K))\subset K$. Without loss of generality, we can assume that $r(A)<r(K)$ and so $z\notin A$. Then we must have $\proj(z,A)=x$ and $d(z,A)=\|z-x\|$. Thus, 
    \[
    d_H(K,A)\geq e(K,A)\geq \|z-x\|=r(K)-r(A),
    \]
    and therefore we obtain that $r(A)\geq r(K)-\frac{r(K)}{2}=\delta>0$.
    
    Now consider $A,B\in \mathcal{D}_Y$ such that $d_H(A,K),d_H(B,K)\leq \delta$ and $u\in\sph_{m-1}$. Suppose without loss of generality that $r_{A}(u) > r_{B}(u)$. We observe that $v:=r_A(u)\cdot u$ belongs to $A$, and so $0<d(v,B)=\|v-p\|\leq d_H(A,B)$, where $p:=\proj(v,B)$. We distinguish two cases. First, $p\in\R u$, see Figure~\ref{fig:TwoCases}. In this case, we deduce that $p = r_B(u)u$ and so
        \[
        r_A(u) - r_B(u) = \|v-p\| = d(v,B) \leq d_H(A,B).
        \]
        Second, $p\notin \R u$. We then can define $w:=r(B)\frac{v-p}{\|v-p\|}$, which, by definition of $r(B)$, verifies $w\in B$. Note that $[0,w]$ and $[p,v]$ are parallel segments, and so $\{0,w,v,p\}$ are the vertices of a trapezoid in the plane generated by $p$ and $v$. Thus the diagonals of this trapezoid $[w,p]$ and $[0,v]$ intersect at a unique point, which we will denote by $b$, see Figure~\ref{fig:TwoCases}. 

    \begin{figure}[h]
\centering
\begin{minipage}{0.45\textwidth}
\centering
			\begin{tikzpicture}[scale=4,>=latex
				]
                \clip (-0.2,-0.1) rectangle (0.7,0.7);
				\draw (0,0) circle (5pt);
                \draw (0,0) circle (12pt);
                \node (B) at (-0.1,0.2) {$B$};
                \node (A) at (0.3,0.4) {$A$};
                \draw[thick,blue] (0,0) -- (0.37,0.2);
                \draw[dashed,red] (-1.2*0.37,-1.2*0.2) -- (1.5*0.37,1.5*0.2) node[below]{\small $\R u$};
                \filldraw[blue] (0.37,0.2) circle (0.4pt) node[above, xshift=3, yshift=1] {$v$};
                \filldraw[blue] (0.42*0.37,0.42*0.2) circle (0.4pt) node[above, xshift=3, yshift=1] {$p$};
			\end{tikzpicture}
\end{minipage}
\hfill
\begin{minipage}{0.45\textwidth}
\centering
\begin{tikzpicture}[scale=4,>=latex]
                
                \clip (-0.2,-0.1) rectangle (0.7,0.7);
                \draw (0,0) circle (12pt);
                \draw (0,0) ellipse (3pt and 10pt);
                \node (B) at (-0.15,0.15) {$B$};
                \node (A) at (0.3,0.4) {$A$};
                \draw[thick,blue] (0,0) -- (0.37,0.2);
                \draw[dashed,red] (0.1,0.16) -- (0.37,0.2);
                \draw[dashed,red] (0,0) -- (0.43*0.27,0.43*0.04) node[right] {$w$};
                \filldraw[red] (0.38*0.27,0.38*0.04) circle (0.4pt); 
                
                \filldraw[blue] (0.37,0.2) circle (0.4pt) node[above, xshift=3, yshift=1] {$v$};
                \filldraw[blue] (0.09,0.16) circle (0.4pt) node[above, xshift=3, yshift=1] {$p$};
                
                \draw[red,thick] (0.38*0.27,0.38*0.04)--(0.09,0.16);
                \filldraw[black] (0.27*0.37,0.27*0.2) circle (0.4pt) node[above left] {$b$};

\end{tikzpicture}
\end{minipage}
\caption{Illustration of $v$, $p = \proj(v,B)$ To the left, the case where $p$ is colinear with $v$. To the right, the construction of $b$.}
\label{fig:TwoCases}
\end{figure}
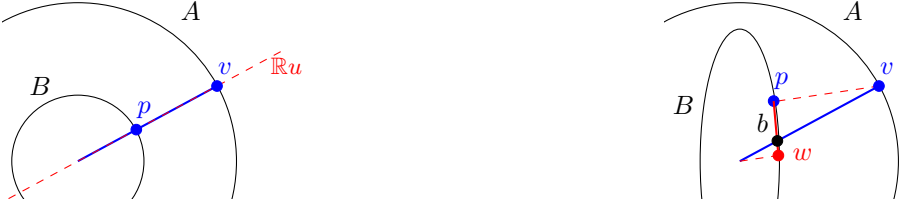
By similarity of triangles we have 
    \[
    \frac{\|v-b\|}{d(v,B)}=\frac{\|b\|}{r(B)}.
    \]
    Since $b\in B$ and  
    $b$ is parallel to $v$, we have that
    $\|b\|\leq r_B(u)$ and so 
    \begin{equation}
    \begin{aligned}    
        r_A(u)-r_B(u)\leq r_A(u)-\|b\|
        &= \|v-b\|\\
        &=  \frac{\|b\|}{r(B)}d(v,B)\\
        &\leq \frac{r_B(u)}{r(B)} d_H(A,B)\leq \frac{\diam(B)}{r(B)} d_H(A,B).
    \end{aligned}
    \end{equation}
    In both cases, we deduce that $r_A(u) - r_B(u) \leq \frac{\diam(Y)}{r(B)} d_H(A,B)$.

    Therefore, a Lipschitz constant for $A\mapsto r_A(u)$ uniformly on $u$ around $K$ is 
    \[ 
    L_{K}:=\sup\left\{\frac{\diam(A)}{r(A)} : {d_{H}(A,K)\leq r(K)/2} \right\} \leq \frac{2\diam(Y)}{r(K)}. 
    \]
\end{proof}

\subsection{Lipschitz selections} Recall that for a set-valued map $T:X\tto Y$, a selection of $T$ is a map $\tau:X\to Y$ verifying that $\tau(x)\in T(x)$ for every $x\in X$. Our aim in this section is to study some Lipschitz selections that we will need in the sequel. We base our developments over the Steiner points, which is a standard tool to produce Lipschitz selections (see, e.g., \cite[Chapter 9]{Aubin2009-ox}). For a set $A\in \D_Y$, we denote by $s_m(A)$ the \emph{Steiner point} (or curvature centroid) of $A$, which is defined by
\begin{equation}
    s_m(A):=\frac{1}{\lambda(\Ball_m)}\int_{\sph_{m-1}}u\sigma_A(u)d\haus_{m-1}(u)
\end{equation}
where $\sigma_A(u):=\sup_{a\in A}\langle u, a\rangle$ is the support functional of $A$. The following proposition shows that $s_m$ is Lipschitz, as a map from $\D_Y$ to $Y$.

\begin{proposition}(\cite[Theorem 9.4.1]{Aubin2009-ox})
\label{prop:steiner selection}
    For every $A\in\D_Y$, $s_{m}(A)\in \ri(A)$. Moreover, the map $s_{m}:(\D_Y,d_H)\to Y$ is $m$-Lipschitz, that is, for every $A,B\in\D_Y$
    \[ 
    \|s_{m}(A)-s_{m}(B)\| \leq m d_{H}(A,B). 
    \]
\end{proposition}

The following lemma can be deduced from \cite[Theorem 9.5.3]{Aubin2009-ox} and allows us to obtain selections that pass through a specific given point in the graph.

\begin{lemma}
\label{lem:lipschitz selections}
    Let $T:X\tto Y\subseteq\R^m$ be a Lipschitz set-valued map whose values are nonempty, convex and compact, and let $(\bar x,\bar y)\in X\times Y$ such that $\bar{y}\in T(\bar{x})$. Then there exists a Lipschitz selection $\tau$ of $T$ with $\Lip(\tau)\leq5m\Lip(T)$ and such that $\tau(\bar x)=\bar y$.
\end{lemma}

We finish this section with a technical lemma, from which we may deduce that when $T(x)$ has constant dimension it is possible to produce locally Lipschitz orthonormal bases of $\mathrm{aff}(T(x))$. 

\begin{lemma}\label{lem:orthogonal basis}
    Let $T:X\tto Y$ be Lipschitz and such that $T(x)$ is convex, compact with $0\in \ri(T(x))$ for each $x\in X$. Let $\bar{x}\in X$ be such that $r(T(\bar{x}))>1$ and let $k:=\dim(T(\bar{x}))$. Then, there exists $\delta>0$ and functions $b_{i}:B(\bar{x},\delta)\to \R^m$, $i=1,\ldots,m$, such that
    \begin{enumerate}
        \item $\forall x\in B(\bar{x},\delta)$ the set $\{ b_{i}(x) \}_{i=1}^{m}$ is an orthonormal basis of $\R^m$,
        \item $\forall x\in B(\bar{x},\delta)$, the set $\{ b_{i}(x) \}_{i=1}^{k}\subseteq T(x)$, and
        \item Every $b_{i}$ is Lipschitz in $B(\bar{x},\delta)$ and $\Lip(b_{i},\bar{x})\leq 5m^3\Lip(T)$.
    \end{enumerate}
\end{lemma}

\begin{proof} 
    Let $\{ \bar{y}_1,\ldots,\bar{y}_{k} \}\subset T(\bar{x})$ be an orthonormal set, which exists by the assumptions. Indeed, from $\dim(T(\bar{x}))=k$ we can take an orthogonal set $\{ \bar{y}_1,\ldots,\bar{y}_{k} \}$ in $\spn(T(\bar{x}))$, the conditions $0\in\ri(T(\bar{x}))$ and $T(\bar{x})$ being convex allows us to have them in $T(\bar{x})$ by a possible scalar multiplication, while $r(T(\bar{x}))>1$ shows that we can take them with $\|\bar{y}_{i}\|=1$, for each $i\in [m]$. 
    
    We complete this orthonormal set to an orthonormal basis of $\R^m$ with vectors $\{ \bar{y}_{k+1},\ldots,\bar{y}_{m}\}$. By virtue of Lemma \ref{lem:lipschitz selections}, for each $i=1,\ldots,k$ we obtain $u_{i}:X\to Y$ Lipschitz selections of $T$ such that 
    \( u_{i}(\bar{x})=\bar{y}_{i}. \)
    For each $i=k+1,\ldots,m$ let $u_{i}:X\to Y$ be the constant function equal to $\bar{y}_{i}$. Note that $\Lip(u_i,\bar{x})\leq 5m\Lip(T)$ for all $i\in[m]$.

    Since all the functions $u_{i}$ are continuous and $\{u_i(\bar{x})\}_{i=1}^{m} = \{\bar{y}_{i}\}_{i=1}^{m}$ is linearly independent, continuity of determinants entails that there exists $\delta>0$ such that for each $x\in B(\bar{x},\delta)$, $\{ u_{i}(x) \}_{i=1}^{m}$ is linearly independent. In particular, for each $x\in B(\bar{x},\delta)$, $\{ u_i(x) \}_{i=1}^{k}$ is a basis for $\spn(T(x))$. For each $x\in B(\bar{x},\delta)$, we apply the Gram-Schmidt orthogonalization procedure to $\{ u_i(x) \}_{i=1}^{m}$ and obtain

    \begin{equation}
        \begin{aligned}
            b_1(x)&:=\frac{u_1(x)}{\|u_1(x)\|},\\
            v_j(x) &:= u_j(x)-\sum_{i=1}^{j-1}\langle u_j(x),b_i(x)\rangle b_i(x),\quad
            b_{j}(x):= \frac{v_j(x)}{\left\| v_j(x) \right\|}.
        \end{aligned}
    \end{equation}
    Using Lemma \ref{lem:lipschitz algebra} we have
    \(
    \Lip(b_{1},\bar{x})\leq \frac{1}{\|u_{1}(\bar{x})\|} \Lip(u_{1},\bar{x}) \leq 5m\Lip(T).
    \) 
    Recall that $u_i(\bar{x}) = b_i(\bar{x}) = \bar{y}_i$, for every $i\in [m]$. Then, for $j>1$ 
    \begin{align*}
        \Lip(v_{j},\bar{x}) &\leq \Lip(u_j,\bar{x}) + \sum\nolimits_{i=1}^{j-1} \Lip(\langle u_j,b_i\rangle b_i,\bar{x})\\
        &\leq 5m\Lip(T)+\sum\nolimits_{i=1}^{j-1} \underbrace{|\langle u_j(\bar{x}),b_i(\bar{x})\rangle|}_{=0} \Lip(b_{i},\bar{x})
        + \|b_{i}(\bar{x})\|\Lip(\langle u_j,b_i\rangle,\bar{x})\\
        &= 5m\Lip(T)+\sum\nolimits_{i=1}^{j-1} \Lip(\langle u_j,b_i\rangle,\bar{x})\\
        &\leq 5m\Lip(T) + \sum\nolimits_{i=1}^{j-1} \|u_j(\bar{x})\|
        \Lip(b_{i},\bar{x}) + \|b_i(\bar{x})\|\Lip(u_j,\bar{x}) \\
        &\leq 5m\Lip(T) + (j-1)5m\Lip(T) + \sum\nolimits_{i=1}^{j-1}\Lip(b_{i},\bar{x})\\
        &= 5mj\Lip(T) + \sum\nolimits_{i=1}^{j-1}\Lip(b_{i},\bar{x}).
    \end{align*}
    and
    \[ \Lip(b_{j},\bar{x}) \leq \frac{1}{\|v_{j}(\bar{x})\|}\Lip(v_{j},\bar{x}) \leq 5mj\Lip(T)+\sum\nolimits_{i=1}^{j-1}\Lip(b_{i},\bar{x}). \]
    Recursively, we deduce 
    \(
    \Lip(b_{j},\bar{x}) \leq 5m\Lip(T)\sum_{i=1}^{j}i \leq 5m^3\Lip(T),
    \)
    for all $j>1$. Hence, the functions $b_{j}$ are Lipschitz in $B(\bar{x},\delta)$, possibly replacing $\delta$ by some smaller radius. Moreover, by construction, for every $x\in B(\bar{x},\delta)$ the set $\{ b_{i}(x) \}_{i=1}^{m}$ is an orthonormal basis of $\R^m$. Notice that since $r(T(\bar{x}))>1$, $\delta$ can be chosen so that for every $x\in B(\bar{x},\delta)$ we have $r(T(x))>1$ and so $\{ b_{i}(x) \}_{i=1}^{k}\subseteq T(x)$.
\end{proof}

\begin{remark}
\label{rem:global orthonormal selection}
    The orthonormal basis Lipschitz local selection $\{b_j\}_{j=1}^m$ obtained in Lemma \ref{lem:orthogonal basis} cannot be extended globally in general. For example, consider $X:=[0,1)$ with the metric $d(x,x')=\min (|x-x'|,1-|x-x'|)$. That is, $(X,d)$ represents a circle (by identifying 1 with 0) with the intrinsic distance. Let $T:X\tto \overline{B}(0,2)\subset\R^2$ defined by $T(x):=2\conv(\gamma(x),-\gamma(x))$ where $\gamma(x):=(\cos(\pi x),\sin(\pi x)))$. We have that $T$ satisfies all the assumptions of Lemma \ref{lem:orthogonal basis}. However, the only continuous local selections $\tau$ of $T$ around $\bar x=\frac{1}{2}$ such that $\|\tau(x)\|=1$ for all $x$, are $\tau=\gamma$ and $\tau=-\gamma$ but both of them have a discontinuity at $0$, when considered globally in the metric space $(X,d)$.
\end{remark}


\section{Lipschitz continuity of neutral belief}
\label{sec:lipschitz beliefs}

This is the main section of our work. As we discussed, our analysis of Lipschitz continuity of $\phi:x\mapsto \E_{\beta_x}[\theta(x,\cdot)]$ is reduced to study the Lipschitz continuity of the neutral belief over $S$, denoted by $\iota:x\in X\mapsto \mathcal{P}(Y)$ (see, cf.~\eqref{eq:def neutral belief}). Our analysis is divided in three cases: 1) when $S(x)$ is full-dimensional (nonempty interior) for every $x\in X$; 2) when $x\mapsto \dim(S(x))$ is constant, not necessarily equal to $m$; and 3) the general case when $x\mapsto \dim(S(x))$ might vary.

\subsection{The full-dimensional case}\label{subsec:full dimensional} The following theorem shows that in the full-dimensional case it is possible to retrieve Lipschitz continuity of the neutral belief with respect to the total variation distance.

\begin{theorem}\label{thm:neutral belief is radon lipschitz}
    Let $S:X\tto Y$ be a set-valued map whose images are convex and compact with nonempty interior and consider $\iota:X\to \P(Y)$ the neutral belief over $S$. If $S$ is continuous, then $\iota$ is continuous with respect to $\dTV$. Moreover, if $S$ is Lipschitz then $\iota$ is locally Lipschitz with respect to $\dTV$ with 
    \[ 
    \Lip_{\mathrm{TV}}(\iota,x) \leq \frac {2L_{Y,m}}{\lambda(S(x))}\Lip(S,x), \quad \forall x\in X, 
    \] 
    where $L_{Y,m}$ is given as in Lemma~\ref{lem:lipvolume}. Moreover, if $X$ is compact then $\iota$ is Lipschitz with respect to $\dTV$ with
    \begin{equation}
    \label{eq:iota global lip} 
    \Lip_{\mathrm{TV}}(\iota)\leq \frac{2L_{Y,m}\Lip(S)}{\min_{x\in X}\lambda(S(x))}.
    \end{equation}
    
\end{theorem}

\begin{proof} 
     Take $x,x'\in X$ and consider a function $f\in L^\infty(Y)$ with $\|f\|_\infty\leq 1$. From Lemma \ref{lem:lipvolume} we obtain 
    \[
    \begin{aligned}
    |\E_{\iota_x}[f]-\E_{\iota_{x'}}[f]| &\leq \left|\E_{\iota_x}[f]-\frac{1}{\lambda(S(x))}\int_{S(x')}f\right|+\left|\frac{1}{\lambda(S(x))}\int_{S(x')}f-\E_{\iota_{x'}}[f]\right|\\
    &\leq \frac{1}{\lambda(S(x))}\left|\int_{S(x)} f-\int_{S(x')}f\right|+\frac{\left|\lambda(S(x'))-\lambda(S(x))\right|}{\lambda(S(x'))\lambda(S(x))}\int_{S(x')}|f|  \\
      & \leq \frac{1}{\lambda(S(x))}\lambda(S(x)\Delta S(x'))\|f\|_\infty+\frac{L_{Y,m}d_H(S(x),S(x'))}{\lambda(S(x))}\|f\|_\infty\\
      &\leq \frac{2L_{Y,m}}{\lambda(S(x))}d_H(S(x),S(x')).
    \end{aligned}
     \]
     Therefore, by the arbitrariness of $f$ we deduce that
     \begin{equation}
     \label{eq:iota radon Lipschitz}
\dTV(\iota_x,\iota_{x'})\leq \frac{2L_{Y,m}}{\lambda(S(x))}d_H(S(x),S(x')).
     \end{equation}
     By the continuity of the composition $\lambda\circ S$, again thanks to Lemma \ref{lem:lipvolume}, we deduce that $\iota$ is continuous with respect to $\dTV$. Additionally, if $S$ is Lipschitz we conclude that $\iota$ is locally Lipschitz with respect to $\dTV$ and that 
     \[
     \Lip_{\mathrm{TV}}(\iota,x)\leq\frac {2L_{Y,m}\cdot \Lip(S,x)}{\lambda(S(x))},\quad\forall x\in X.
     \]
     Finally, if $X$ is compact then $\bar{v}:=\inf_{x\in X} \lambda(S(x))>0$, and \eqref{eq:iota radon Lipschitz} entails that $\iota$ is globally Lipschitz with the bound \eqref{eq:iota global lip}.
\end{proof}

The continuity and Lipschitz continuity of the neutral belief with respect to the total variation distance seem to require full-dimensionality of the images of $S$. The following simple example illustrates this idea.

\begin{example}
\label{ex:counterexample lower dim}
    Consider the set-valued map $S:[0,1]\tto [0,1]^2$ given by $S(x)=\{x\}\times[0,1]$. Then for $f\in \CC([0,1]^2)$ defined by $f(y):=\sqrt{y_1}$ we obtain
    \[
    \E_{\iota_x}[f]=\sqrt{x},
    \]
    which clearly is not Lipschitz. Therefore, by virtue of Proposition \ref{prop:reduction to neutral beliefs}, the neutral belief $\iota$ over $S$ is not Lipschitz with respect to $\dTV$.

    Moreover, it is possible to prove that $\iota$ is not even continuous with respect to the topology induced by $\dTV$. Indeed, take $f_n(y)=\sqrt[n]{y_1}$ from which we see that for any $x>0$
    \[
    \dTV(\iota_x,\iota_0)\geq |\E_{\iota_x}[f_n]-\E_{\iota_0}[f_n]|=\sqrt[n]{x}\to 1,\text{ as } n\to\infty. 
    \]
    Hence, the belief $\iota$ is not continuous (at $0$) with respect to $\dTV$.\hfill$\Diamond$   
\end{example}

\begin{remark}
To better understand the intuition behind the requirement of full dimensional images in Theorem 4.1, we observe that in the full dimensional setting small enough perturbations of a reference set have an intersection with the reference set of a large proportion. In contrast, for a set of lower dimension there are arbitrarily small perturbations that have no single point in common with the reference one, and therefore the probability mass may be drastically redistributed. This is exactly what occurs in Example \ref{ex:counterexample lower dim}.    
\end{remark}

\subsection{The case of constant dimension}
\label{subsec:constant dim}
In this subsection we aim to prove a result analogue to Theorem \ref{thm:neutral belief is radon lipschitz}, but now relaxing the assumption of nonempty interior of the images of the set-valued map $S$. Instead, we assume that $\dim(S(x))=k$ for all $x\in X$ (the images of $S$ have constant dimension) and obtaining the (local) Lipschitz property of the neutral belief, here with respect to the Wasserstein-1 distance.

\begin{proposition}\label{prop:loclipkvol}
    Let $S:X\tto Y$ be Lipschitz and such that $S(x)$ is convex, compact with $\dim(S(x))=k$ for each $x\in X$. Then $x\mapsto \lambda_{k}(S(x))$ is locally Lipschitz. In addition, if $X$ is compact, then $\lambda_k\circ S$ is Lipschitz.
\end{proposition}
\begin{proof}
Let $\bar{x}\in X$ and let us prove that $\lambda_k\circ S$ is Lipschitz in a ball around $\bar{x}$. First, we will consider the particular case where $S$ satisfies the following assumptions
\begin{enumerate}
    \item[(a)] $0\in\ri (S(x))$ for all $x\in X$, and
    \item[(b)] there exists $\delta>0$ such that $S(x) \subseteq\R^k\times \{0\}^{m-k}$ for all $x\in B(\bar x,\delta)$.
\end{enumerate}
In this case, by considering the canonical isometry between $\R^k\times\{0\}^{m-k}$ and $\R^k$ we may consider $S$ as with full-dimensional values (with nonempty interior) in $\R^k$ and apply Lemma \ref{lem:lipvolume} to conclude that $\lambda_k\circ S$ is Lipschitz in $B(\bar x,\delta)$. 

In the general case, we shall show that we can define, by means of translations and rotations over $S$, another convex compact set $\tilde{Y}$ and another set-valued map $U:X\tto \tilde{Y}$ verifying (a) and (b), together with the rest of assumptions in the present proposition. Hence, $\lambda_k\circ U$ is Lipschitz in the ball $B(\bar x,\delta)$ (from the previous case) and this yields that the same property holds for $S$, since $\lambda_k$ is invariant with respect to translations and rotations.

Let us set $\tilde{Y}:=\overline{B}(0,\diam(Y))$. We can define the set-valued map $\tilde{S}:X\tto \tilde{Y}$, given by $\tilde{S}(x):=S(x)-s_m(S(x))$ for all $ x\in X$. Thanks to Proposition \ref{prop:steiner selection}, $\tilde{S}$ is also Lipschitz with nonempty convex and compact images (subsets of $\tilde{Y}:=\overline{B}(0,\diam(Y))$) such that $\dim(\tilde{S}(x))=k$ and $0=s_m(S(x))-s_m(S(x))=s_m(\tilde{S}(x))\in \ri(\tilde{S}(x))$ for all $x\in X$. Since $\tilde{S}$ is constructed through translations of the images of $S$, the volume is preserved. Therefore, by replacing $S$ with $\tilde{S}$ if necessary, we may assume from now on without loss of generality that $S$ satisfies assumption (a).

We claim that we can assume without loss of generality that $r(S(\bar{x}))>1$. Indeed, we note that $r(S(\bar{x}))>0$, and so we can define $\kappa:=(1+r(S(\bar{x}))^{-1})>0$ and the set-valued map $T:X\tto Y$ given by $T(x):=\kappa S(x)$ for $x\in X$. We observe that $T$ is Lipschitz with $\Lip(T)=\kappa \Lip (S)$, it has convex and compact values and satisfies $$r(T(\bar{x}))=\kappa r(S(\bar x))=r(S(\bar x))+1>1.$$ Therefore,  noting that $\lambda_k\circ T=\kappa^k(\lambda_k\circ S)$ we deduce that the Lipschitz property of $\lambda_k\circ T$ implies that of $\lambda_k\circ S$. This justifies that we may assume $r(S(\bar x)>1$.

Now, consider $\delta>0$ and the functions $b_{i}:B(\bar{x},\delta)\to \R^m$ associated with this set-valued map $S$ as given in Lemma \ref{lem:orthogonal basis}. Let $L$ be a common Lipschitz constant for the functions $b_{i}$. Notice that the matrix-valued function $P:B(\bar{x},\delta)\to\mathcal{M}_{m\times m}(\R)$ given by $P(x):= \begin{bmatrix}b_{1}(x)& \cdots &b_{n}(x)\end{bmatrix}$ is Lipschitz with respect to the distance associated to the operator norm in the space of matrices $\mathcal{M}_{m\times m}(\R)$. Indeed, for every $x,x'\in B(\bar{x},\delta)$ and $z\in \Ball_m$ we have that 
\[ \|(P(x)-P(x'))z\| \leq \sum_{i=1}^{m}\|b_{i}(x)-b_{i}(x')\||z_{i}| \leq \left(\sum_{i=1}^{m} |z_{i}| \right) Ld(x,x')\leq \sqrt{m} Ld(x,x'), \]
which shows that $d(P(x),P(x'))\leq\sqrt{m}Ld(x,x')$.

Define the set-valued map $U:X\tto\tilde{Y}$ given by $U(x):=P(x)^{\top} S(x)$ for each $x\in B(\bar{x},\delta)$. For each $x\in B(\bar x,\delta)$, since $P(x)^\top$ acts as a rotation of $S(x)$, then the set $U(x)$ is also convex compact, $\dim(U(x))=k$ and $0\in \ri(U(x))$. Moreover, $U$ is Lipschitz on $B(\bar{x},\delta)$. Indeed,
\[
\begin{aligned}
    d_H(U(x),U(x')) &\leq d_H(P(x)^{\top}S(x),P(x')^{\top}S(x)) + d_H(P(x')^{\top}S(x),P(x')^{\top}S(x'))\\
    &\leq \|P(x) - P(x')\|\|S(x)\| + \|P(x')\|d_H(S(x),S(x')),
\end{aligned}
\]
where $\|S(x)\| = \sup\{\|y\|\colon y\in S(x)\}$. Since $P(x')$ is unitary, its operator norm verifies that $\|P(x')\|=1$. Then, 
we get that
\[
\begin{aligned}
d_H(U(x),U(x')) &\leq \diam(Y)\|P(x) - P(x')\| + d_H(S(x),S(x'))\\
&\leq (\sqrt{m}\diam(Y)L+\Lip(S))d(x,x').
\end{aligned}
\]
Clearly, the images of $U$ are nonempty convex and compact and satisfies the assumptions (a) and (b), so we may deduce from the first part of this proof that $\lambda_k\circ U$ is Lipschitz in $B(\bar x,\delta)$.

Finally, since $P(x)$ is a unitary matrix (rotation) we conclude $\lambda_k\circ S$ is Lipschitz in $B(\bar x,\delta)$ by the rotation invariance of $\lambda_k$ (see, e.g., \cite[Theorem 2]{evans2015measure}).
\end{proof}

We now present the main theorem of Lipschitz continuity for the case where $S$ has constant affine dimension. 

\begin{theorem}
\label{thm:Lip Wasserstein1}
    Let $S:X\tto Y$ be a set-valued map such that $S(x)$ is convex compact with $\textup{dim}(S(x))=k$ for each $x\in X$. If $S$ is Lipschitz, then the neutral belief $\iota$ is locally Lipschitz with respect to $\dW$ with
    $$\Lip_{\dW}(\iota,x)\leq 
 \frac{C}{\lambda_k(S(x))^2}$$
for some constant $C$ depending on $\diam(Y)$, $m$, $k$ and $\Lip(S)$. Moreover, if $X$ is compact, then $\iota$ is Lipschitz with respect to $\dW$.
\end{theorem}

\begin{proof}
Let $\bar{x}\in X$. 
    As in the proof of Proposition~\ref{prop:loclipkvol}, without loss of generality we can suppose that for any $x\in X$ it holds $s_m(S(x))=0$ so that $0\in \ri(S(x))$, and also that $r(S(\bar x))>1$. 
  
    Consider $\delta>0$ and $P:B(\bar{x},\delta)\to\mathcal{M}_{m\times m}(\R)$ as in the proof of Proposition~\ref{prop:loclipkvol}. 
    Let $f:Y\to\R$ with $\Lip(f)\leq 1$. Noting that
    \[ 
    |\E_{\iota_{x}}[f]-\E_{\iota_{x'}}[f]| =  |\E_{\iota_{x}}[f-f(0)]-\E_{\iota_{x'}}[f-f(0)]|, 
    \]
    
    We assume without lossing generality that $f(0) = 0$. Using change of variables, the goal now is to write $\int_{S(x)}f(z)dz$ as an integral in $\Ball_k$. Recalling the notation $r_{A}(u):=\sup\{t\geq 0: tu\in A\}$ and $U(x) := P(x)^{\top}S(x)$, we see that
    \[
    \begin{aligned}
    \int_{S(x)} f(z) dz &= \int_{U(x)} f(P(x)z)\underbrace{|\det(P(x))|}_{=1} dz \\ 
    &= \int_{\mathbb{S}_{k-1}}\int_{0}^{r_{U(x)}(v)} f(tP(x)v,0_{m-k}) t^{k-1} dt dv \\
    &= \int_{\mathbb{S}_{k-1}}\int_{0}^{1} f(r_{U(x)}(v)tP(x)v,0_{m-k}) r_{U(x)}(v)^{k} t^{k-1} dt dv \\
    &=\int_{\Ball_k} f\left(r_{U(x)}\left(\frac{z}{|z|}\right)P(x)z,0_{m-k}\right)r_{U(x)}\left(\frac{z}{|z|}\right)^{k} dz. \end{aligned}
    \]
    We claim that for fixed $z\in \Ball_k\setminus\{0\}$, the function
   \[ 
   \varphi_{z}(x):=  f\left(r_{U(x)}\left(\frac{z}{|z|}\right)P(x)z,0_{m-k}\right)r_{U(x)}\left(\frac{z}{|z|}\right)^{k}
   \]
    is Lipschitz near $\bar{x}$, with constant independent of $z$. Indeed, by identifying $U(x)\subset \R^k\times\{0_{m-k}\}$ as a full dimensional subset of $\R^k$, Lemma \ref{lem:inner radius lip} entails that $x\mapsto r_{U(x)}(\frac{z}{|z|})$ is Lipschitz over $B(\bar{x},\delta)$ uniformly in $z$. 
    Therefore applying the calculus rules in Lemma \ref{lem:lipschitz algebra}, we deduce the functions $\varphi_{z}$, $z\in \Ball_k\setminus\{0\}$, are Lipschitz over $B(\bar{x},\delta)$, with a common Lipschitz constant $K>0$. Furthermore, note that for every $v\in \sph_{k-1}$ one has that $r_{U(x)}(v)\leq r_Y(P(x)v)\leq \diam(Y)$, and so
    \[
    \varphi_z(x) \leq \diam(Y)\max_{v\in \sph_{k-1}}\{ r_{U(x)}(v)^k \}\leq \diam(Y)^{k+1}, \qquad \forall z\in\Ball_k\setminus\{0\}.
    \]
    Finally, by Proposition \ref{prop:loclipkvol}, $x\mapsto \lambda_k(S(x))$ is Lipschitz near $\bar{x}$, with a constant $K_{Y,m}\Lip(S)$  where $K_{Y,m}>0$ depends only on $\diam(Y)$ and $m$. Then, noting that we can assume $\tfrac{1}{2}\lambda_k(S(\bar{x}))\leq \lambda_k(S(x))\leq \diam(Y)^k\lambda_k(\Ball_k)$ for $x\in B(\bar{x},\delta)$, we have that
    \[ 
    \begin{aligned}
    |\E_{\iota_{x}}[f]-\E_{\iota_{x'}}[f]| &=\left| \int_{\Ball_k} \frac{\varphi_{z}(x)}{\lambda_k(S(x))}-\frac{\varphi_{z}(x')}{\lambda_k(S(x'))} dz\right|\\
    &= \left| \int_{\Ball_k} \frac{\varphi_{z}(x)-\varphi_z(x')}{\lambda_k(S(x))}+\left(\frac{1}{\lambda_k(S(x))} - \frac{1}{\lambda_k(S(x'))}\right)\varphi_z(x') dz\right|\\
    &\leq \lambda_k(\Ball_k) \left( \frac{K}{\lambda_k(S(x))} + \frac{K_{Y,m}\Lip(S)}{\lambda_k(S(x))\lambda_k(S(x'))}\diam(Y)^{k+1}\right)d(x,x')\\
    &\leq \lambda_k(\Ball_k) \left( \frac{2K\lambda_k(S(\bar{x}))}{\lambda_k(S(\bar{x}))^2} + \frac{4K_{Y,m}\Lip(S)}{\lambda_k(S(\bar{x}))^2}\diam(Y)^{k+1}\right)d(x,x')\\
    &= \frac{C}{\lambda_k(S(\bar{x}))^2}d(x,x').
    \end{aligned}
    \]
    Since this last estimate is independent of $f$ as long as it is 1-Lipschitz, we deduce that $x\mapsto \iota_{x}$ is locally Lipschitz with respect to  $\dW$. 
    
    The proof is then complete, since the second part of the statement is direct. 
\end{proof}

\begin{remark}
    In contrast to Section~\ref{subsec:full dimensional} concerning the full-dimensional case, assuming compactness of $X$ in Proposition~\ref{prop:loclipkvol} and Theorem~\ref{thm:Lip Wasserstein1}, we do not have explicit bounds for the global Lipschitz constants of the volume function and the neutral belief over $S$, respectively. We may obtain explicit bounds under the additional assumption that $X$ is a quasiconvex space. 
    Considering Remark \ref{rem:global orthonormal selection} the technique developed in this paper cannot directly be applied to get explicit bounds on the Lipschitz constants
\end{remark}

\subsection{The general case of variable dimension}
\label{subsec:noncanstant dim}

 Motivated by \cite{salas2023existence}, we aim to study the variation of dimensionality of $S(x)$ by 
 approximating it by inner and outer ``rectangles''. That is, we look at the existence of set-valued maps $T_0,T_1,R_0,R_1:X\tto \R^m$ such that
 \(
 T_0(x) + R_0(x) \subset S(x) \subset T_1(x) + R_1(x),
 \) 
 and the additional property that $T_1$ and $T_2$ have constant affine dimension while $R_0$ and $R_1$ control the variation of dimension. In \cite{salas2023existence}, continuity of the neutral belief is deduced as a consequence of rectangular continuity: that is, continuity of $T_0$ and $T_1$, and a balancing relation of the volumes of $R_0$ and $R_1$. However, Example~\ref{ex:lip svm nonlip belief} shows that this is not enough for Lipschitz continuity. By reinforcing the hypotheses on the maps $T_0, T_1,R_0$ and $R_1$, we deduce the following theorem.

\begin{theorem}\label{thm:rectangular}
Let $S:X\tto Y$ be a set valued map with $S(x)$ nonempty convex and compact for each $x\in X$. Assume that there exists a constant $L>0$ such that for all $\bar{x}\in X$ there exists $\delta>0$ and set-valued maps $T_0,T_1:X\tto Y$ and $R_0,R_1:X\tto \R^m$ have nonempty convex and compact values such that 
\begin{enumerate}
\item[(i)] For every $x\in B(\bar{x},\delta)$,
\begin{equation}\label{eq:sandwich}
    T_0(x)+R_0(x)\subseteq S(x)\subseteq T_1(x)+R_1(x).
\end{equation}
\item[(ii)] For $j=0,1$, $R_j(x)\subseteq \spn(T_j(x)-T_j(x))^{\perp}\cap\spn(S(x)-S(x))$, and $T_j$ has constant affine dimension, i.e. $\dim(T_j(x))=\dim(T_j(\bar{x}))$  for all $x\in B(\bar{x},\delta)$.
\item[(iii)] For $j=0,1$, $T_j$ and $R_j$ are $L$-Lipschitz with 
$R_j(\bar{x})=\{0\}$ and $T_j(\bar{x})=S(\bar{x})$.

\item[(iv)] The 
function
\begin{equation}\label{eq:ProportionOrthogonalPart}
h(x):=\left\{\begin{array}{cl}\displaystyle\frac{\lambda_{d_x-d_{\bar{x}}}(R_1(x))\lambda_{d_{\bar{x}}}(T_1(x))}{\lambda_{d_x-d_{\bar{x}}}(R_0(x)) \lambda_{d_{\bar{x}}}(T_0(x))},&\text{if }x\in B(\bar{x},\delta)\setminus\{\bar{x}\}\\
\\
1&\text{if } x=\bar{x}.\\
\end{array}\right.
\end{equation} 
is $L$-Lipschitz, where $d_x:=\dim(S(x))$.
\end{enumerate}
Then the neutral belief $\iota$ over $S$ is calm with respect to $\dW$.
Moreover, if $X$ is a compact 
quasiconvex space, then $\iota$ is Lipschitz with respect to $\dW$.
\label{thm:belief W1 Lip nonconstant dim}
\end{theorem}

\begin{proof}
Fix $\bar{x}\in X$ and we shall prove that $\iota$ is calm at $\bar{x}$. We first assume that $\dim (S(\bar{x}))=k$ and $\dim(S(x))=l>k$ for all $x\neq \bar{x}$ near enough $\bar{x}$. Let us write $r:=l-k>0$.

Take $f:Y\to\R$ Lipschitz with $\Lip(f)\leq 1$ and assume without loss of generality that $\min_Y f=0$. Let us denote $\|R_1(x)\|:=\sup_{z_r\in R_1(x)}\|z_r\|$. Then for $x\neq \bar{x}$
\begin{equation*}
    \begin{aligned}
        \frac{1}{\lambda_{r}(R_1(x))}\int_{S(x)}fd\lambda_{l}&\leq \frac{1}{\lambda_{r}(R_1(x))}\int_{T_1(x)+R_1(x)}f(z)d\lambda_{l}(z)\\
        &=\frac{1}{\lambda_{r}(R_1(x))}\int_{T_1(x)}\int_{R_1(x)}f(z_t+z_r)d\lambda_r(z_r)d\lambda_{k}(z_t)\\
        &=\int_{T_1(x)}\left(\frac{1}{\lambda_{r}(R_1(x))}\int_{R_1(x)}f(z_t+z_r)d\lambda_r(z_r)\right)d\lambda_{k}(z_t)\\
        &\leq \int_{T_1(x)} (f(z_t)+\|R_1(x)\|)d\lambda_{k}(z_t)\\
        &=\int_{T_1(x)} fd\lambda_{k}+\|R_1(x)\|\lambda_k(T_1(x))\\
    \end{aligned}
\end{equation*}
Then, we have that
\begin{equation*}
\begin{aligned}
    \E_{\iota_{x}}[f]
    &\leq \frac{1}{\lambda_k(T_0(x))\lambda_r(R_0(x))}\int_{S(x)} fd\lambda_l\\
    &\leq \frac{\lambda_k(T_1(x))\lambda_r(R_1(x))}{\lambda_k(T_0(x))\lambda_r(R_0(x))}\left(\frac{1}{\lambda_k(T_1(x))}\int_{T_1(x)} fd\lambda_k+\|R_1(x)\|\right)\\
\end{aligned}
\end{equation*}
where we recognize one of the terms as the expected value of $f$ with respect to the uniform distribution over 
$T_1(x)$. 
Therefore noting that $d_H(R_1(\bar{x}),R_1(x))=\|R_1(x)\|\leq Ld(x,\bar{x})$ we have
\begin{equation}
\label{eq:ineq iota1}
\begin{aligned}
    \E_{\iota_{x}}[f]&\leq h(x)\left(\E_{\iota^1_{x}}[f]+Ld(x,\bar{x}) \right),
\end{aligned}
\end{equation}
where $\iota^1$ is the neutral belief over $T_1$. By Theorem \ref{thm:Lip Wasserstein1} we know that for some $L_1>0$ we have
\[
\E_{\iota^1_x}[f]-\E_{\iota_{\bar{x}}}[f]=\E_{\iota^1_{x}}[f]-\E_{\iota^1_{\bar{x}}}[f]\leq L_1 d(x,\bar{x})
\]
Then, using the Lipschitz continuity of $h$, and assuming $\delta<L^{-1}$ we get
\[
\begin{aligned}
\E_{\iota_x}[f]-\E_{\iota_{\bar{x}}}[f]&\leq h(x)L_1 d(x,\bar x)+(h(x)-1)\E_{\iota_{\bar{x}}}[f]+Lh(x)d(x,\bar{x}) \\
&\leq ((L+L_1)h(x)+L\E_{\iota_{\bar{x}}}[f]) d(x,\bar{x})\\
&\leq ((L+L_1)(1+L\delta)+L\|f\|_\infty)d(x,\bar x)\leq (2L+2L_1+\|f\|_\infty)d(x,\bar{x})
\end{aligned}
\]
Using a similar argument (now based on the Lipschitz continuity of $T_0$ instead of $T_1$) we may obtain a bound for $\E_{\iota_x}[f]-\E_{\iota_{\bar{x}}}[f]$. Indeed, in the same vein of \eqref{eq:ineq iota1} we can prove that
\begin{equation*}
\begin{aligned}
    \E_{\iota_{x}}[f]&\geq 
    h(x)^{-1}\left(\frac{1}{\lambda_k(T_0(x))}\int_{T_0(x)} fd\lambda_k-\|R_0(x)\|\right)\\
\end{aligned}
\end{equation*}
and noting that $\|R_0(x)\|\leq Ld(x,\bar{x})$ we have
\begin{equation}
\label{eq:ineq iota iota0}
\begin{aligned}
    \E_{\iota_{x}}[f]&\geq h(x)^{-1}\left(\E_{\iota^0_{x}}[f]-Ld(x,\bar{x}) \right),
\end{aligned}
\end{equation}
where $\iota^0$ is the neutral belief over $T_0$. Again using Theorem \ref{thm:Lip Wasserstein1}, we know that there exists $L_0>0$ such that
\begin{equation}
\label{eq:ineq combined expected}
\E_{\iota_{\bar{x}}}[f]-\E_{\iota^0_x}[f]=\E_{\iota^0_{\bar{x}}}[f]-\E_{\iota^0_{x}}[f]\leq L_0 d(x,\bar{x}).
\end{equation}

From the fact that $\lim_{x\to \bar{x}} h(x)=1$, taking $\delta<(2L)^{-1}$ and by a nonlocal analogue of Lemma \ref{lem:lipschitz algebra} we see that $\Lip(1/h(\cdot))\leq\frac{L}{(1-L\delta)^2}\leq 4L$. Using this together \eqref{eq:ineq iota iota0} and \eqref{eq:ineq combined expected} with may prove that
\[
\begin{aligned}
\E_{\iota_{\bar{x}}}[f]-\E_{\iota_x}[f]
&\leq (4L\|f\|_\infty+2(L+L_0))d(x,\bar x)
\end{aligned}
\]

Summing up we deduce that if $\delta>0$ is small enough, there exists a constant $\widehat{L}>0$ such that
\[
|\E_{\iota_x}[f]-\E_{\iota_{\bar{x}}}[f]|\leq \widehat{L} d(x,\bar{x}).
\]
Since $\Lip(f)\leq1$ and $\min f=0$ implies $\|f\|_\infty\leq\diam(Y)$, we can take for instance
$$\widehat{L}= 2L+4L\diam(Y)+2\max\{L_1,L_0\}.$$
Since this is true for all $f:X\to\R$ with $\Lip(f)\leq 1$ we conclude that $\widehat{L}$ is an upper bound for the modulus of calmness for $\iota$ with respect to $\dW$.

Now for the general case note that continuity of $S$ entails that there is $\delta>0$ such that $k=
\dim(S(\bar{x}))\leq\dim(S(x))$ for all $x\in B(\bar{x},\delta)$. Then we can write $B(\bar{x},\delta)=\bigcup_{l= k}^{m}X_l$ where $X_l:=\bar{x}\cup \{x\in B(\bar{x},\delta): \dim(S(x))=l\}$. We have that $S$ restricted to $X_l$ satisfies all the properties of the theorem so that we conclude that the restriction of $\iota$ to $X_l$ is calm at $\bar{x}$ and hence, also in the union $X$, since it is finite.

Finally, if $X$ is quasiconvex, then thanks to Lemma \ref{lem:punctual to local lip} and since the modulus of calmness are uniformly bounded, we deduce that $\iota$ is Lipschitz.
\end{proof}

We observe that the assumptions (i) to (iv) in Theorem \ref{thm:belief W1 Lip nonconstant dim} are satisfied if $S$ is $L$-Lipschitz and the images have constant dimension around the reference point $\bar{x}$. This follows from taking the set-valued maps $T_0:=T_1:=S$ and $R_0:=R_1:=\{0\}$. Therefore, Theorem~\ref{thm:belief W1 Lip nonconstant dim} is a generalization of Theorem \ref{thm:Lip Wasserstein1}, in the setting of quasiconvex spaces.

\section{Applications to bilevel programming}\label{sec:bilevel}

We study the applications to two standard settings in bilevel programming: 1) when $S(x)$ is given by approximated solutions of a lower-level problem verifying Slater CQ; and 2) when $S(x)$ is given as the exact solution set of a parametric fully linear problem.

\subsection{Approximated solutions under Slater CQ}
\label{subsec:slatercq}

Let us consider a (regularized) bilevel programming problem of the form
\begin{equation}
\label{eq:bilevel regularized}
\begin{aligned}
    \min_{x\in X}&\quad\theta(x,y)\\
    s.t.&\quad  y\in \varepsilon\text{-}\argmin_{z}\{f(x,z):g(x,z)\leq 0\}.
\end{aligned}
\end{equation}
In this model only $x$ is decided by the leader while $y$ is the decision of the follower and it is modeled by the leader as a random variable with support in $S(x):=\varepsilon\text{-}\argmin_{z}\{f(x,z):g(x,z)\leq 0\}$. A simple way to deal with the uncertainty, called the Bayesian approach in \cite{salas2023existence}, is that the leader has a belief of its distribution and pose the problem as in equation \eqref{eq:Decision-Dependent-Uncertainty}. 

\begin{theorem}\label{thm:Lip-SlaterCase}
Let $X\subset\R^n$ be a nonempty set, $\varepsilon>0$ and for each $x\in X$ consider $S(x):=\varepsilon\text{-}\argmin_y \{f(x,y):g(x,y)\leq 0\}$. We assume
\begin{enumerate}
    \item[(i)] $f:\R^n\times\R^m\to\R$ and $g:\R^n\times\R^m\to\R^p$ are locally Lipschitz functions and for each $x\in X$,
    \item[(ii)]$f(x,\cdot)$ and $g_i(x,\cdot)$ are convex  for all $i\in[p]$,
    \item[(iii)]Slater CQ holds, that is, $\exists y\in\R^m$ such that $g_i(x,y)<0$ for all $i\in[p]$.
\end{enumerate}
Assume also that the set
\[
\hat{Y} = \bigcup_{x\in X} \{y\in \R^m\colon g(x,y)\leq 0\},
\]
is compact. Then the neutral belief over $S$ is locally Lipschitz with respect to the total variation distance $\dTV$.
\end{theorem}

\begin{proof}
First note that since the conclusion is about local Lipschitz continuity in a subset of $\R^n$, we may assume that $X:=\overline{B}(x_0,\delta_0)$ for some $x_0\in X$ and $\delta_0>0$. Since in this case $X$ is compact, any local Lipschitz map, as $f$ and $g$, is therefore Lipschitz.

The hypotheses ensure that $S:X\tto Y$ has convex compact values with nonempty interior, where $Y$ can be taken as the closed convex hull of $\hat{Y}$. Moreover, we claim that $S$ is Lipschitz. In view of Theorem \ref{thm:neutral belief is radon lipschitz}, this implies that the neutral belief $\iota$ is locally Lipschitz with respect to $\dTV$. So let us prove our claim.

Let us assume first that $f$ is a constant function so that $S(x)$ can be written as $$S(x)=\{y\in\R^m:h(x,y)\leq 0\},$$
where $h(x,y):=\max_{i=1}^pg_i(x,y).$ Due to our assumptions, $h$ is (locally) Lipschitz and for each $x\in X$, $h(x,\cdot)$ is convex and satisfies Slater CQ, that is, there exists $y\in \R^m$ such that $h(x,y)<0$. Then, for every $x\in X$ and for any $y$ such that $h(x,y)\geq 0$ it holds $0\notin \partial_y h(x,y)$, where $\partial_y h(x,y)$ stands for the usual convex subdifferential of $h(x,\cdot)$ (see, e.g., \cite{rockafellar2009variational}). We know that convexity of $h(x,\cdot)$ and continuity of $h$ entails, as a mild application of Attouch theorem \cite{attouch1977convergence}, that the slope function $(x,y)\mapsto d(0,\partial_y h(x,y))$ is lower semicontinuous (see \cite{DaniilidisSalasTapia-Garcia2025}). If we let $\bar{x}\in X$, using the compactness of $\{y\colon h(\bar{x},y) =0\}$ and the monotonicity of the slope along steepest descent curves (see, e.g., \cite[Theorem 17.2.3]{Attouch2014Variational}), we deduce that there exist $\delta,\alpha>0$ such that $d(0,\partial_y h(x,y))\geq \alpha>0$ for all $x\in B(\bar{x},\delta)$ and $y\in \R^m$ such that $h(x,y)\geq 0$. Then taking $\gamma=\frac{2}{\alpha}>0$ (see e.g. \cite{fabian2010error}) we obtain an error bound
$$d(y,S(x))\leq \gamma\max\{ h(x,y),0\},\quad \forall x\in B(\bar x,\delta),\forall y\in \R^m.$$
If we take $y\in S(x')$, so that $h(x',y)\leq 0$, then 
\[
\begin{aligned}
    d(y,S(x))&\leq \gamma \max\{h(x,y)-h(x',y)+h(x',y),0\}\\
    &\leq \gamma \max\{h(x,y)-h(x',y),0\}\\
    &\leq \gamma|h(x,y)-h(x',y)|\leq \gamma L\|x-x'\|.
\end{aligned}
\]
Hence taking supremum over $y\in S(x')$ we obtain
$e(S(x'),S(x))\leq \gamma L\|x-x'\|$ and by symmetry we deduce that $S$ is Lipschitz in $B(\bar{x},\delta)$ with Lipschitz constant $\gamma L$.

Next we consider the general case, that is, when $f$ is not necessarily constant and we shall see that this case can be reduced to the previous case. Indeed, we define 
\[K(x):=\{y\in \R^m:g(x,y)\leq0\}\]
which from the previous analysis it can be deduced that $K$ is Lipschitz. This together with the Lipschitz continuity of $f$ implies that the value function 
\[
v(x):=\inf_y\{f(x,y):y\in K(x)\}
\] 
is (locally) Lipschitz. Indeed, let $\bar{x}\in X$, take $\delta>0$, $x,x'\in B(\bar{x},\delta)$ and $\varepsilon>0$. Then, there exists $y\in K(x)$ such that $v(x)+\varepsilon \geq f(x,y)$. Since $K$ is Lipschitz with respect to the Hausdorff distance we know there exists $y'\in K(x')$ such that $\|y-y'\|\leq \Lip(K)d(x,x')$. Then we have
\[
\begin{aligned}
v(x')-v(x) - \varepsilon&\leq f(x',y')-f(x,y)\\
&\leq \Lip(f)\cdot(\|(x',y')-(x,y)\|)\\
&\leq \Lip(f) (\|x'-x\|+\|y'-y\|)\leq \Lip(f) (1+\Lip(K))\|x'-x\|
\end{aligned}
\]
Taking $\varepsilon\to 0$, we deduce by symmetry that $v$ is Lipschitz in $B(\bar{x},\delta)$. Therefore, $g_0(x,y):=f(x,y)-v(x)-\varepsilon$ defines a Lipschitz function, convex on the second variable, and satisfying the Slater CQ. Moreover, we may write 
$$S(x):=\{y\in\R^m: \tilde{h}(x,y)\leq 0\}$$
where $\tilde{h}(x,y)=\max_{i=0}^p g_i(x,y)$, from which we see that $S$ is Lipschitz, and the proof is complete.
\end{proof}

\begin{corollary}
\label{cor:bilevel Slater}
    Under the assumptions of Theorem \ref{thm:Lip-SlaterCase} and for the neutral belief, the regularized bilevel problem \eqref{eq:bilevel regularized} formulated as \eqref{eq:Decision-Dependent-Uncertainty} under the Bayesian approach, has a locally Lipschitz objective function.
\end{corollary}

\subsection{Exact solutions in linear bilevel problems}

In this section we consider the model \eqref{eq:bilevel regularized} but with exact solutions in the lower level problem and a fully linear structure, that is,
\begin{equation}
\begin{aligned}
    \min_{x\in X}&\quad g^{\top}x+h^{\top}y\\
    s.t.&\quad  y\in \argmin_{z}\{c^{\top}z\colon Ax+Bz\leq b\},
\end{aligned}
\end{equation}
where $X:=\{x\in \R^n:\exists y\in \R^m, Ax+By\leq b\},$ and $A,B$ and $b,c,g,h$ are matrices and vectors of appropriate dimensions.
As in Section \ref{subsec:slatercq}, the leader's problem is as \eqref{eq:Decision-Dependent-Uncertainty} and can be simplified to 
\begin{equation}
\label{eq:bilevel linear expected}
\begin{aligned}
    \min_{x\in X}&\quad g^{\top}x+\E_{\beta_x}[h^{\top}y]
\end{aligned}
\end{equation}
where $\beta_x\in \P(Y)$ concentrates on $S(x):=\argmin_{z}\{c^{\top}z\colon Ax+Bz\leq b\}$ for $x\in X$.

\begin{lemma}\label{lemma:Hoffman-uniform} Let $U\in \mathcal{M}_{m\times q}(\R)$. For every $w\in \R^m$, let $F(w) := \{y\mid Uy\leq w\}$. For every $k\in\N$, there exists a constant $H(U,k)>0$  such that for every collection $w_1,\ldots,w_k\in\R^m$ one has that
\[
\bigcap_{i=1}^k F(w_i) \neq \emptyset \implies 
d\left(y, \textstyle\bigcap_{i=1}^k F(w_i) \right) \leq H(U,k)\max_{i\in[k]} d(y,F(w_i)), ~\forall y\in\R^q. 
\]
\end{lemma}
\begin{proof}
    By \cite{Hoffman1952}, for every matrix $M\in \mathcal{M}_{m_1\times m_2}(\R)$ and every vector $b\in\R^{m_1}$ such that the system $Mz\leq b$ is consistent, there exists a constant $c>0$ such that
    \[
    d(x, \{z:Mz\leq b\}) \leq c\|(Mx-b)_+\|_{\infty}, \quad\forall x\in \R^{m_2},
    \]
     where, for $a\in\R^{m_1}$, $a_+:=(\max\{a_1,0\},\ldots,\max\{a_{m_1},0\})$. By \cite[Proposition 1]{PennaVeraZuluaga2021}, the constant $c$ can be taken as a constant $H(M)$ depending only on the matrix $M$. Set $M\in \mathcal{M}_{km\times q}(\R).
    $ as the matrix that has $k$ copies of $U$ downwards 
and $b\in\R^{km}$ as the vector obtained by concatenating $w_1,\ldots, w_k$. We have that
the nonemptyness of $\bigcap_{i=1}^k F(w_i)$ ensures that the system $Mz\leq b$ is consistent. Consequently, we obtain
\[
d\left(y, \textstyle\bigcap_{i=1}^k F(w_i) \right)=\; d(y, \{z:Mz\leq b\})\leq H(M)\|(My-b)_+\|_{\infty}, \quad\forall y\in \R^q.
\]   
The conclusion follows by noting that 
    \[
    \begin{aligned}
    \|(My-b)_+\|_{\infty} &= \max_{i\in[k]}\|(Uy-w_i)_+\|_{\infty}\\
    &\leq \max_{i\in[k]}\min_{z\in F(w_i)}\|(Uy-Uz)_+\|_{\infty} \leq \|U\|_{*}\max_{i\in[k]} d(y, F(w_i)),
    \end{aligned}
    \]
    where $\|U\|_* = \sup\{\|Uz\|_{\infty}\colon \|z\|_{\infty} = 1\}$. Then, it is enough to define $H(U,k) = \|U\|_*H(M)$.
\end{proof}

\begin{theorem}\label{thm:Lip-LinearCase} Assume that $D:=\{(x,y)\in\R^n\times\R^m: Ax+By\leq b\}$ is nonempty and bounded. Consider $S:X\tto Y$ given by $S(x):=\argmin_y \{c^\top y:Ax+By\leq b\}$, where $X:=\{x:\exists y\in \R^m, (x,y)\in D\}$ and $Y:=\{ y:\exists x\in \R^n, (x,y)\in D\}$, and let $\iota$ be the neutral belief over $S$. Then $\iota$ is Lipschitz with respect to $\dW$.
\end{theorem}

\begin{proof}
We know that $S:X\tto Y$ has convex compact values and it is Lipschitz (see e.g. \cite[Chapter IX, section 7]{dontchev2006well}), say $L$-Lipschitz.
Note that incorporating the optimality as a new constraint we may write $S(x) = \{y\colon\tilde{B}y \leq \varphi(x)\}$, where $\tilde{B}\in\mathcal{M}_{(p+1)\times m}(\R)$ and $\varphi:X\to\R^{p+1}$ is Lipschitz.  Let $\bar{x}\in X$ and $k:=\dim(S(\bar{x}))$ and let us assume without loss of generality that $0\in S(\bar{x})$.

Let $F := \mathrm{span}(S(\bar{x}))$ and define the set-valued maps $R:X\tto \R^m$ and $T_j:X\tto \R^m$ given by
\[
\begin{aligned}
R(x) := \mathrm{proj}(S(x); F^{\perp}),\quad
T_1(x):=\proj(S(x),F), \quad
T_0(x) := \bigcap_{z\in R(x)} (S(x) - z).
\end{aligned}
\]
From the construction we see that 
\begin{equation*}
T_0(x)+R(x)\subseteq S(x)\subseteq T_1(x)+R(x),\qquad \forall x\in X.
\end{equation*}
Moreover, $R$ and $T_1$ are $L$-Lipschitz as composition of Lipschitz maps. We shall prove that 
$T_0$ is also Lipschitz. Since, $R(x)$ is a compact polytope, we have that the set of extreme points $\mathrm{ext}(R(x))$ is nonempty and finite, and we can write 
\begin{equation}\label{eq:T0 simplified}
T_0(x) = \bigcap_{z\in \mathrm{ext}(R(x))} S(x)-z.
\end{equation}
Indeed, the direct inclusion holds. Now, let $y\in \bigcap_{z\in \mathrm{ext}(R(x))} S(x)-z$. This yields that $y+z\in S(x)$ for all $z\in \mathrm{ext}(R(x))$. Now, let $r \in R(x)$. Then, there exists nonnegative values $(t_z\colon z\in\mathrm{ext}(R(x)))$ such that
    \begin{equation*}
    r = \sum_{z\in \mathrm{ext}(R(x))} t_z z\quad\text{ and }\quad\sum_{z\in \mathrm{ext}(R(x))} t_z = 1.
    \end{equation*}
    Thus, $y + r = \sum_{z\in \mathrm{ext}(R(x))} t_z (y+z) \in S(x)$ by convexity. Then $y\in S(x)-r$, and since $r\in R(x)$ is arbitrary, we conclude that $y\in T_0(x)$. This proves \eqref{eq:T0 simplified}.

We know by \cite{salas2023existence} that $T_0(x)$ is nonempty for every $x$ in some neighborhood $U$ of $\bar{x}$ in $X$. Since $S(x)$ has at most $N=\binom{p}{m}$ extreme points, we can define
    \(
    \kappa := \max_{k\in [N]} H(\tilde{B},k),
    \)
    where $H(\tilde{B},k)$ is given by Lemma \ref{lemma:Hoffman-uniform}. Then, since $|\mathrm{ext}(R(x))|\leq N$ and noting that $S(x) - z = \{w:\tilde{B}w\leq \varphi(x)-\tilde{B}z \}$, we get by Lemma \ref{lemma:Hoffman-uniform} that for all $y \in \R^m$
    \begin{equation*}
    \begin{aligned}
    d(y,T_0(x)) 
    &=d\left(y,\bigcap_{z\in\mathrm{ext}(R(x))}\{w:\tilde{B}w\leq \varphi(x)-\tilde{B}z \}\right)\\
    &\leq \kappa \max_{z\in \mathrm{ext}(R(x))}d\left(y,\{w:\tilde{B}w\leq \varphi(x)-\tilde{B}z \}\right)\\
    &= \kappa \max_{z\in \mathrm{ext}(R(x))} d(y,S(x)-z).
    \end{aligned}
    \end{equation*}
    Now, applying the formula above for every $y\in S(\bar{x})$ we get that
    \[
    \begin{aligned}
    d(y,T_0(x)) &\leq \kappa 
    \max_{z\in \mathrm{ext}(R(x))} d(y,S(x)-z)\\
    &\leq \kappa \left(\max_{z\in\mathrm{ext}(R(x))}\|z\|+d(y,S(x))\right)\\
    &\leq \kappa(d_H(R(x),R(\bar{x})))+d_H(S(x),S(\bar{x})) \leq (\kappa +1)L\|x-\bar{x}\|.
    \end{aligned}
    \]
    Noting that 
    \begin{align*}
    \sup_{y\in T_0(x)} d(y,S(\bar{x})) &\leq \sup_{y\in T_0(x), r\in R(x)}\|r\| + d(y+r,S(\bar{x}))\\
    &\leq d_H(R(x), R(\bar{x})) + \sup_{y\in S(x)}d(y,S(\bar{x}))\leq 2L\|x-\bar{x}\|,
    \end{align*}
    we conclude that 
    \(
    d_H(T_0(x), S(\bar{x})) \leq 2\kappa L\|x-\bar{x}\|.
    \)
    and so $T_0$ is $2\kappa L$-calm at $\bar{x}$. Since $\bar{x}$ is arbitrary (and so $T_0$ is uniformly calm), and since $X$ is a geodesic space by convexity, we deduce that $T_0$ is Lipschitz by Lemma~\ref{lem:punctual to local lip}.
    The conclusion of the theorem follows directly from Theorem~\ref{thm:rectangular}.
\end{proof}

\begin{corollary}
\label{cor:linear bilevel}
    For the neutral belief, the objective function of the linear bilevel problem under the Bayesian approach \eqref{eq:bilevel linear expected} is Lipschitz relative to its domain.
\end{corollary}

\begin{remark}
In the linear case, the intuition is that the polyhedral structure of $S$ ensures that dimension changes occur in a controlled manner, which compensates for the lack of full dimensionality.    
\end{remark}

\section{Conclusion and open questions}\label{sec:conclusion}

The main contribution of this paper is providing sufficient conditions for the (locally) Lipschitz property of the expected value in the case of decision-dependent distributions whose support is convex and compact and moves in a Lipschitz fashion. This is done by (and reduced to) studying the Lipschitz property of the neutral belief over Lipschitz set-valued maps.

The present work is limited to prove the Lipschitz property of the expected value but not necessarily to give sharp Lipschitz constants. While some explicit bounds are obtained for the full-dimensional case or under quasiconvexity of the space, the general question of computation of Lipschitz constants and how to exploit them in algorithms is out of the present scope, and we leave it open for a future work.


\paragraph{Acknowledgements:} This paper has been partially supported by the project Math-Amsud 23-MATH-13. The third and fourth authors were partially supported by the project FONDECYT 1251159, ANID-Chile. The third author  was partially funded by the Center
of Mathematical Modeling (CMM), FB210005, BASAL funds for centers of excellence
(ANID-Chile). The second author was partially supported by Universidad de O’Higgins (112011002-PD-17706171-9). We would like to thank Professor Terry Rockafellar for suggesting us to use the name Lipschitz continuity instead of Lipschitzianity, a term we used in a preliminary version of the manuscript.

\bibliographystyle{plain}
\bibliography{references.bib}

@book{cobzacs2019lipschitz,
  title={Lipschitz functions},
  author={Cobza{\c{s}}, {\c{S}}tefan and Miculescu, Radu and Nicolae, Adriana and others},
  volume={2241},
  year={2019},
  publisher={Springer}
}

@incollection{serra1998hausdorff,
    author={Serra, Jean},
    title = {Hausdorff distances and interpolations},
    booktitle = {Mathematical morphology and its applications to image processing},
    publisher = {Kluwer Academic Publishers},
    volume={12},
    pages={107--114},
    year = {1998},
}

@book{evans2015measure,
    AUTHOR = {Evans, Lawrence C. and Gariepy, Ronald F.},
     TITLE = {Measure theory and fine properties of functions},
    SERIES = {Textbooks in Mathematics},
   EDITION = {Revised},
 PUBLISHER = {CRC Press, Boca Raton, FL},
      YEAR = {2015},
     PAGES = {xiv+299},
      ISBN = {978-1-4822-4238-6},
   MRCLASS = {28-01},
  MRNUMBER = {3409135},
}

@BOOK{Aubin2009-ox,
  title     = "Set-valued analysis",
  author    = "Aubin, Jean-Pierre and Frankowska, H{\'e}l{\`e}ne",
  publisher = "Birkh{\"a}user",
  series    = "Modern Birkh{\"a}user Classics",
  edition   =  1990,
  month     =  mar,
  year      =  2009,
  address   = "Cambridge, MA",
  language  = "en"
}

@article{buttazzo1995minimum,
  title={Minimum problems over sets of concave functions and related questions},
  author={Buttazzo, Giuseppe and Ferone, Vincenzo and Kawohl, Bernhard},
  journal={Mathematische Nachrichten},
  volume={173},
  number={1},
  pages={71--89},
  year={1995},
  publisher={Wiley Online Library}
}

@book{dempe2002foundations,
  title={Foundations of bilevel programming},
  author={Dempe, Stephan},
  year={2002},
  publisher={Springer}
}

@article{attouch1977convergence,
  title={Convergence de fonctions convexes, des sous-diff{\'e}rentiels et semi-groupes associ{\'e}s},
  author={Attouch, H{\'e}dy},
  journal={CR Acad. Sci. Paris},
  volume={284},
  number={539-542},
  pages={13},
  year={1977}
}

@article{dempe2018pessimistic,
  title={Pessimistic bilevel linear optimization},
  author={Dempe, Stephan and Luo, G and Franke, Susanne},
  journal={Journal of Nepal Mathematical Society},
  volume={1},
  number={1},
  pages={1--10},
  year={2018}
}

@book{rockafellar2009variational,
  title={Variational analysis},
  author={Rockafellar, R Tyrrell and Wets, Roger J-B},
  volume={317},
  year={2009},
  publisher={Springer Science \& Business Media}
}

@article{jung1901ueber,
  title={Ueber die kleinste Kugel, die eine r{\"a}umliche Figur einschliesst.},
  author={Jung, Heinrich},
  journal={Journal f{\"u}r die reine und angewandte Mathematik (Crelles Journal)},
  volume={1901},
  number={123},
  pages={241--257},
  year={1901},
  publisher={De Gruyter Berlin, New York}
}

@book{klenke2013probability,
  title={Probability theory: a comprehensive course},
  author={Klenke, Achim},
  year={2013},
  publisher={Springer Science \& Business Media}
}

@book{dontchev2006well,
  title={Well-posed optimization problems},
  author={Dontchev, Assen L and Zolezzi, Tullio},
  year={2006},
  publisher={Springer}
}

@article{fabian2010error,
  title={Error bounds: necessary and sufficient conditions},
  author={Fabian, Marian J and Henrion, Ren{\'e} and Kruger, Alexander Y and Outrata, Ji{\v{r}}{\'\i} V},
  journal={Set-Valued and Variational Analysis},
  volume={18},
  number={2},
  pages={121--149},
  year={2010},
  publisher={Springer}
}

@article{munoz2024exploiting,
    AUTHOR = {Mu\~noz, Gonzalo and Salas, David and Svensson, Anton},
     TITLE = {Exploiting the polyhedral geometry of stochastic linear
              bilevel programming},
   JOURNAL = {Math. Program.},
  FJOURNAL = {Mathematical Programming},
    VOLUME = {210},
      YEAR = {2025},
    NUMBER = {1-2},
     PAGES = {695--730},
       DOI = {10.1007/s10107-024-02097-w}
}

@article{salas2023existence,
  title={Existence of solutions for deterministic bilevel games under a general Bayesian approach},
  author={Salas, David and Svensson, Anton},
  journal={SIAM Journal on Optimization},
  volume={33},
  number={3},
  pages={2311--2340},
  year={2023},
  publisher={SIAM}
}

@article {Hoffman1952,
    AUTHOR = {Hoffman, Alan J.},
     TITLE = {On approximate solutions of systems of linear inequalities},
   JOURNAL = {J. Research Nat. Bur. Standards},
  FJOURNAL = {J. Research Nat. Bur. Standards},
    VOLUME = {49},
      YEAR = {1952},
     PAGES = {263--265},
   MRCLASS = {27.0X},
  MRNUMBER = {51275},
MRREVIEWER = {R.\ Bellman},
}

@article {PennaVeraZuluaga2021,
    AUTHOR = {Pe\~na, Javier and Vera, Juan C. and Zuluaga, Luis F.},
     TITLE = {New characterizations of {H}offman constants for systems of
              linear constraints},
   JOURNAL = {Math. Program.},
  FJOURNAL = {Mathematical Programming},
    VOLUME = {187},
      YEAR = {2021},
    NUMBER = {1-2},
     PAGES = {79--109},
      ISSN = {0025-5610,1436-4646},
   MRCLASS = {90C05 (90C25 90C57)},
  MRNUMBER = {4246299},
MRREVIEWER = {Mahdi\ Moeini},
       DOI = {10.1007/s10107-020-01473-6}
}

@article {DaniilidisSalasTapia-Garcia2025,
    AUTHOR = {Daniilidis, Aris and Salas, David and Tapia-Garc\'ia,
              Sebasti\'an},
     TITLE = {A slope generalization of {A}ttouch theorem},
   JOURNAL = {Math. Program.},
  FJOURNAL = {Mathematical Programming},
    VOLUME = {212},
      YEAR = {2025},
    NUMBER = {1},
       DOI = {10.1007/s10107-024-02108-w}
}

@article {DrusvyatskiyXiao2023:Optimization,
    AUTHOR = {Drusvyatskiy, Dmitriy and Xiao, Lin},
     TITLE = {Stochastic optimization with decision-dependent distributions},
   JOURNAL = {Math. Oper. Res.},
  FJOURNAL = {Mathematics of Operations Research},
    VOLUME = {48},
      YEAR = {2023},
    NUMBER = {2},
     PAGES = {954--998},
      ISSN = {0364-765X,1526-5471},
   MRCLASS = {90C15 (90C25)},
  MRNUMBER = {4588947},
MRREVIEWER = {Yun-Zhi\ Yan},
       DOI = {10.1287/moor.2022.1287}
}

@article {CutlerDiazDrusvyatskiy2024:Approximation,
    AUTHOR = {Cutler, Joshua and D\'iaz, Mateo and Drusvyatskiy, Dmitriy},
     TITLE = {Stochastic approximation with decision-dependent
              distributions: asymptotic normality and optimality},
   JOURNAL = {J. Mach. Learn. Res.},
  FJOURNAL = {Journal of Machine Learning Research (JMLR)},
    VOLUME = {25},
      YEAR = {2024},
     PAGES = {Paper No. [90], 49},
      ISSN = {1532-4435,1533-7928},
   MRCLASS = {90C15 (49J40 62L20 90C33)},
  MRNUMBER = {4749126},
}

@article {CutlerDrusvyatskiyHarchaoui2023:Drift,
    AUTHOR = {Cutler, Joshua and Drusvyatskiy, Dmitriy and Harchaoui, Zaid},
     TITLE = {Stochastic optimization under distributional drift},
   JOURNAL = {J. Mach. Learn. Res.},
  FJOURNAL = {Journal of Machine Learning Research (JMLR)},
    VOLUME = {24},
      YEAR = {2023},
     PAGES = {Paper No. [147], 56},
      ISSN = {1532-4435,1533-7928},
   MRCLASS = {90C15 (90C25)},
  MRNUMBER = {4596094},
}

@article {WoodDallanese2023:StochasticSaddle,
    AUTHOR = {Wood, Killian and Dall'Anese, Emiliano},
     TITLE = {Stochastic saddle point problems with decision-dependent
              distributions},
   JOURNAL = {SIAM J. Optim.},
  FJOURNAL = {SIAM Journal on Optimization},
    VOLUME = {33},
      YEAR = {2023},
    NUMBER = {3},
     PAGES = {1943--1967},
      ISSN = {1052-6234,1095-7189},
   MRCLASS = {90C25 (90C15 90C33 90C47)},
  MRNUMBER = {4622976},
MRREVIEWER = {Jiaming\ Liang},
       DOI = {10.1137/22M1488077}
}

@misc{HeBolognaniDorflerMuehlebach2025Dynamics,
      title={Decision-Dependent Stochastic Optimization: The Role of Distribution Dynamics}, 
      author={Zhiyu He and Saverio Bolognani and Florian Dörfler and Michael Muehlebach},
      year={2025},
      howpublished={e-print ArXiv:2503.07324}
}

@ARTICLE{Wang2025:COnstrained,

  author={Wang, Zifan and Liu, Changxin and Parisini, Thomas and Zavlanos, Michael M. and Johansson, Karl H.},
  journal={IEEE Transactions on Automatic Control}, 
  title={Constrained Optimization With Decision-Dependent Distributions}, 
  year={2025},
  volume={70},
  number={8},
  pages={5176-5189},
  doi={10.1109/TAC.2025.3540441}}

@article {Jonsbraten1998:Class,
    AUTHOR = {Jonsbr{\aa}ten, Tore W. and Wets, Roger J.-B. and Woodruff, David L.},
     TITLE = {A class of stochastic programs with decision dependent random elements},
   JOURNAL = {Ann. Oper. Res.},
  FJOURNAL = {Annals of Operations Research},
    VOLUME = {82},
      YEAR = {1998},
     PAGES = {83--106},
       DOI = {10.1023A:1018943626786}
}

@Inbook{Mallozzi1996,
author="Mallozzi, L.
and Morgan, J.",
title="Hierarchical Systems with Weighted Reaction Set",
bookTitle="Nonlinear Optimization and Applications",
year="1996",
publisher="Springer US",
address="Boston, MA",
pages="271--282",
}

@article{MallozziMorgan2005,
    AUTHOR = {Mallozzi, L. and Morgan, J.},
     TITLE = {Oligopolistic markets with leadership and demand functions
              possibly discontinuous},
   JOURNAL = {Journal of Optimization Theory and Applications},
  FJOURNAL = {Journal of Optimization Theory and Applications},
    VOLUME = {125},
      YEAR = {2005},
    NUMBER = {2},
     PAGES = {393--407},
      ISSN = {0022-3239},
   MRCLASS = {91B42 (91A65)},
  MRNUMBER = {2136502},
MRREVIEWER = {D. A. Novikov},
}

@article {GoelGrossmann2006:Class,
    AUTHOR = {Goel, Vikas and Grossmann, Ignacio E.},
     TITLE = {A class of stochastic programs with decision dependent
              uncertainty},
   JOURNAL = {Math. Program.},
  FJOURNAL = {Mathematical Programming. A Publication of the Mathematical
              Programming Society},
    VOLUME = {108},
      YEAR = {2006},
    NUMBER = {2-3},
     PAGES = {355--394},
      ISSN = {0025-5610,1436-4646},
   MRCLASS = {90C15 (90C11)},
  MRNUMBER = {2238707},
MRREVIEWER = {W.\ K.\ Klein Haneveld},
       DOI = {10.1007/s10107-006-0715-7}
}

@phdthesis{ahmed2000strategic,
  title={Strategic planning under uncertainty: Stochastic integer programming approaches},
  author={Ahmed, S.},
  year={2000},
  school={University of Illinois at Urbana-Champaign}
}

@article {Nohadani2018:Optimization,
    AUTHOR = {Nohadani, Omid and Sharma, Kartikey},
     TITLE = {Optimization under decision-dependent uncertainty},
   JOURNAL = {SIAM J. Optim.},
  FJOURNAL = {SIAM Journal on Optimization},
    VOLUME = {28},
      YEAR = {2018},
    NUMBER = {2},
     PAGES = {1773--1795},
      ISSN = {1052-6234,1095-7189},
   MRCLASS = {90C05 (90C31)},
  MRNUMBER = {3814028},
MRREVIEWER = {Sorin-Mihai\ Grad},
       DOI = {10.1137/17M1110560}
}

@article {BeckLjubicSchmidt2023Survey,
    AUTHOR = {Beck, Yasmine and Ljubi\'c, Ivana and Schmidt, Martin},
     TITLE = {A survey on bilevel optimization under uncertainty},
   JOURNAL = {European J. Oper. Res.},
  FJOURNAL = {European Journal of Operational Research},
    VOLUME = {311},
      YEAR = {2023},
    NUMBER = {2},
     PAGES = {401--426},
      ISSN = {0377-2217,1872-6860},
   MRCLASS = {90-02 (90C17 90C30 90C47 91A65)},
  MRNUMBER = {4610745},
       DOI = {10.1016/j.ejor.2023.01.008}
}

@incollection {BurtscheidtClaus2020BilevelLinear,
	AUTHOR = {Burtscheidt, Johanna and Claus, Matthias},
	TITLE = {Bilevel linear optimization under uncertainty},
	BOOKTITLE = {Bilevel optimization---advances and next challenges},
	SERIES = {Springer Optim. Appl.},
	VOLUME = {161},
	PAGES = {485--511},
	PUBLISHER = {Springer, Cham},
	YEAR = {[2020] \copyright 2020},
	DOI = {10.1007/978-3-030-52119-6\_17},
}

@article {Claus2021Continuity,
	AUTHOR = {Claus, Matthias},
	TITLE = {On continuity in risk-averse bilevel stochastic linear
	programming with random lower level objective function},
	JOURNAL = {Oper. Res. Lett.},
	FJOURNAL = {Operations Research Letters},
	VOLUME = {49},
	YEAR = {2021},
	NUMBER = {3},
	PAGES = {412--417},
	DOI = {10.1016/j.orl.2021.04.007},
}

@article {Claus2022Existence,
	AUTHOR = {Claus, Matthias},
	TITLE = {Existence of solutions for a class of bilevel stochastic
	linear programs},
	JOURNAL = {European J. Oper. Res.},
	FJOURNAL = {European Journal of Operational Research},
	VOLUME = {299},
	YEAR = {2022},
	NUMBER = {2},
	PAGES = {542--549},
	DOI = {10.1016/j.ejor.2021.12.004},
}

@article{BuchheimHenkeIrmai2022Knapsack,
	title = {The Stochastic Bilevel Continuous Knapsack Problem with Uncertain Follower’s Objective},
	author = {Buchheim, C. and Henke, D. and Irmai, J.},
	year = {2022},
	volume = {194},
	pages = {521--542},
	doi = {10.1007/s10957-022-02037-8},
	journal = {J Optim Theory Appl},
	fjournal = {Journal of Optimization Theory Applications}
}

@misc{LiuZhou2023HJmetricspaces,
      title={Hamilton-Jacobi equations in metric spaces}, 
      author={Qing Liu and Xiaodan Zhou},
      year={2023},
      howpublished={e-print arXiv:2308.08073}
}

@book {AmbrosioGigliSavare2008,
    AUTHOR = {Ambrosio, Luigi and Gigli, Nicola and Savar\'e, Giuseppe},
     TITLE = {Gradient flows in metric spaces and in the space of
              probability measures},
    SERIES = {Lectures in Mathematics ETH Z\"urich},
   EDITION = {Second},
 PUBLISHER = {Birkh\"auser Verlag, Basel},
      YEAR = {2008},
     PAGES = {x+334},
      ISBN = {978-3-7643-8721-1},
   MRCLASS = {49-02 (28A33 35K55 35K90 49Q20 60B05)},
  MRNUMBER = {2401600},
MRREVIEWER = {Pietro\ Celada},
}

@book {Attouch2014Variational,
    AUTHOR = {Attouch, Hedy and Buttazzo, Giuseppe and Michaille, G\'erard},
     TITLE = {Variational analysis in {S}obolev and {BV} spaces},
    SERIES = {MOS-SIAM Series on Optimization},
    VOLUME = {17},
   EDITION = {Second},
 PUBLISHER = {Society for Industrial and Applied Mathematics (SIAM),
              Philadelphia, PA; Mathematical Optimization Society,
              Philadelphia, PA},
      YEAR = {2014},
     PAGES = {xii+793},
      ISBN = {978-1-611973-47-1},
       DOI = {10.1137/1.9781611973488}
}

@article{Xiao2026Developing,
author = {Xiao, Nachuan and Ding, Kuangyu and Hu, Xiaoyin and Toh, Kim-Chuan},
title = {Developing Lagrangian-Based Methods for Nonsmooth Nonconvex Optimization},
journal = {Mathematics of Operations Research},
pages = {1-34},
year = {2026},
doi = {10.1287/moor.2024.0479}
}

\end{document}